\documentclass[10pt]{article}
\usepackage[utf8]{inputenc}   
\usepackage[left=2.5cm,right=2.5cm,top=2.5cm,bottom=2.5cm]{geometry}
\usepackage[T1]{fontenc}                
\usepackage[english]{babel}              
\usepackage{graphicx}
\usepackage{amsmath,amsfonts,amssymb,amsthm}  
\usepackage{xspace}
\usepackage{geometry} 
\usepackage[overload]{empheq}
\usepackage{dsfont} 
\usepackage{stmaryrd}
\usepackage{tikz} 
\usepackage{mathrsfs}
\usepackage{enumerate}
\usepackage{subfig}
\usepackage{float}

\usetikzlibrary{arrows}

\usepackage{hyperref}
\hypersetup{
colorlinks = true,
urlcolor = blue,
linkcolor = blue,
citecolor = blue,
}

\newtheorem{theorem}{Theorem}
 
\newtheorem{proposition}{Proposition}
\newtheorem{lemma}{Lemma}
\newtheorem{corollary}{Corollary}

\usepackage{amsmath,amssymb,euscript ,yfonts,psfrag,latexsym,dsfont,graphicx,bbm,color,amstext,wasysym,epsfig,subfig,parskip,textcomp}
\usepackage{amsmath}

\usepackage{pgfplots}
\usepackage{mathrsfs}
\usetikzlibrary{arrows}

\usepackage{natbib}
\bibpunct{(}{)}{;}{a}{,}{,}
\usepackage{hyperref}
\hypersetup{
colorlinks = true,
urlcolor = blue,
linkcolor = blue,
citecolor = blue,
}

\newtheorem{assumption}{Assumption}
\DeclareMathOperator*{\argmax}{arg\,max}

\DeclareMathOperator*{\argmin}{arg\,min}

\makeatletter
\def\thm@space@setup{%
    \thm@preskip=6pt plus 2pt minus 2pt 
    \thm@postskip=1pt plus 2pt minus 2pt 
}
\makeatother
\title{\textbf{Regularity of the score function in generative models}}
\usepackage{authblk}
\author{Arthur St\'ephanovitch}
\affil{\small D\'epartement de Math\'ematiques et Applications\\
       Ecole Normale Sup\'erieure, Université PSL, CNRS\\
       F-75005 Paris, France,\\ \texttt{stephanovitch@dma.ens.fr}}

\date{}

\begin{document}

\maketitle

\begin{abstract}
\noindent We study the regularity of the score function in score-based generative models and show that it naturally adapts to the smoothness of the data distribution. Under minimal assumptions, we establish Lipschitz estimates that directly support convergence and stability analyses in both diffusion and ODE-based generative models. In addition, we derive higher-order regularity bounds, which simplify existing arguments for optimally approximating the score function using neural networks.
\end{abstract}
\tableofcontents

\section{Introduction}
\subsection{Score-based generative modeling}

Generative models aim to learn complex probability distributions from data, allowing for the generation of realistic samples that resemble the underlying distribution. In recent years, a class of generative models known as \emph{score-based generative models} (SGMs) or \emph{diffusion models} \citep{song2020score} has gained significant attention due to their impressive empirical performance in image \citep{dhariwal2021diffusion}, audio \citep{kong2020diffwave}, and scientific data generation tasks \citep{hoogeboom2022equivariant}. These models rely on modeling and approximating the score function---i.e., the gradient of the log-density---of a distribution evolving through a stochastic process.

The key idea behind score-based diffusion models is to define a \emph{forward process} that gradually transforms the data distribution into a simple reference distribution through a stochastic differential equation (SDE). This process is often given by an Itô SDE of the form:
\begin{equation}\label{eq:generalforward}
\mathrm{d} \overrightarrow{X}_t
=
- \lambda \gamma_t \overrightarrow{X}_t \mathrm{d} t 
+
\sqrt{2 \gamma_t }\sigma \mathrm{d} B_t, \quad X_0 \sim p,
\end{equation}
where \( p \) is the data distribution, $\sigma > 0$ a diffusion coefficient, $\lambda > 0$ a drift coefficient  and  $\gamma : \mathbb{R}_{\geq 0} \to \mathbb{R}_{>0}$ a deterministic locally integrable noise schedule such that $\int_0^\infty \gamma_s \mathrm{d} s = \infty$. Given a final time $T>0$, one can constructs a \emph{backward process}: 
\begin{equation}\label{eq:truebackward}
\mathrm{d}\overleftarrow{X}_t = \Big(\lambda\gamma_{T-t} \overleftarrow{X}_t  + ( \sigma^2 \gamma_{T-t} + b_t^2 ) \nabla \log \overrightarrow{p}_{T-t}(\overleftarrow{X}_t)\Big) \mathrm{d}t + \sqrt{2}b_t \, \mathrm{d}\overline{B}_t,\quad \overleftarrow{X}_0 \sim \overrightarrow{p}_T,
\end{equation}
where \( \overrightarrow{p}_t \) is the law of $\overrightarrow{X}_t$,  $b : [0,T] \to \mathbb{R}_{\geq 0}$ is a deterministic  backward diffusion schedule and \( \overline{B}_t \) a Brownian motion independent of $B_t$. The idea behind this backward process is that under minimal conditions, it can be shown that the law of $
X_{T-t}
$ is $\overrightarrow{p}_t$ \citep{ANDERSON1982313}. Therefore, if one can approximate the score function
\begin{equation}\label{eq:scfunc2}
    s^\star(t,x):= \nabla \log \overrightarrow{p}_t(x).
\end{equation}
by a certain $\hat{s}(x,t)$,
and if $\overrightarrow{p}_T$ is close enough to a certain $\overrightarrow{p}_\infty$ being an easy-to sample distribution, one can generate new approximate samples by numerically approximating the equation
\begin{equation}\label{eq:theestimator}
\mathrm{d}Y_t = \Big(\lambda\gamma_{T-t} Y_t  + ( \sigma^2 \gamma_{T-t} + b_t^2 ) \hat{s}_{T-t}(Y_t)\Big) \mathrm{d}t + \sqrt{2}b_t \, \mathrm{d}\overline{B}_t,\quad Y_0 \sim \overrightarrow{p}_\infty.
\end{equation}
This general formulation encompasses several popular models. For instance, \emph{Denoising Diffusion Probabilistic Models} (DDPM) take \( b_t = \sigma \), while  \emph{Probability Flow ODE} \citep{song2020score}, use \( b_t = 0 \). Other sampling schemes interpolate between these two extremes by choosing hybrid schedules for \( b_t \)  \citep{karras2022elucidating}.

Recent theoretical advances have analyzed the statistical and computational convergence of SGMs. Several works---including \cite{lee2023convergence, chen2022sampling,pmlr-v202-chen23q,wibisono2024optimal} in the case of diffusion, and \cite{chen2023probability,benton2023error,huang2024convergence} in the case of ODE---provide convergence guarantees for score-based models under regularity conditions on the score function. These include global \emph{Lipschitz continuity} and \emph{one-sided Lipschitz} conditions (bound only on the largest eigenvalue of the Jacobian), which enable the application of refined Grönwall-type inequalities for bounding the divergence between trajectories of different SDEs.

Despite these advances, the precise regularity properties of the score function remain underexplored. Understanding and quantifying the \emph{Lipschitz regularity} of the score function is not only crucial for guaranteeing the convergence of diffusion samplers but also for designing numerically stable and sample-efficient algorithms. This paper addresses this gap by rigorously analyzing the Lipschitz behavior of the score function with respect to time under minimal assumptions on the data distribution $p$.

In addition to these Lipschitz estimates, we also provide bounds on higher-order derivatives of the score, showing that the regularity of the score adapts to the smoothness of the underlying data distribution. In particular, we establish Hölder-type bounds on successive derivatives, which reflect how the forward flow progressively amplifies the regularity of the initial log  density. These results are then leveraged to obtain significantly shorter  proofs that the score function can be efficiently approximated using neural networks.

\subsection{Previous works and contributions}
Recent works \citep{brigati2024heat,gentiloni2025beyond,stephanovitch2024smooth} have investigated the regularity properties of the score function in score-based generative models.  In particular, \cite{brigati2024heat} provides sharp dimension-free estimates on the Jacobian of the score function for log-Lipschitz perturbations of log-concave measures. Their analysis is used to establish the Lipschitz regularity of the transport map associated to the heat flow.

In this work, we aim to extend these results in two directions. First, in Theorem \ref{theo:mainlipreguofscore} we relax the regularity assumptions on the data distribution by allowing log-\( \beta \)-Hölder perturbations for any \( \beta \in (0,1] \). This result allows in particular to obtain in Proposition \ref{prop:liptran} that the transport map built by the probability flow ODE is Lipschitz. Second, in the case where the data distribution has a Gaussian-type behavior at infinity, we show in Corollary \ref{coro:lipschitzbound} that the estimates on the operator norm of the Jacobian of the score can actually be improved (in terms of time dependency).

The work by \citet{gentiloni2025beyond} also derives Lipschitz estimates on the score function. However, their analysis relies on stronger structural assumptions on the data distribution. In particular, they assume that the gradient of the log-density satisfies a one-sided Lipschitz condition, which excludes important cases such as compactly supported measures.

On the other hand, \cite{stephanovitch2024smooth} establishes high-order regularity results for the score function and the associated transport map in the Langevin flow setting. The analysis shows that, under sufficient smoothness of the data distribution, the score inherits this regularity, allowing for control over higher-order derivatives of the transport. Specifically, it is considered data distributions that are either log-Hölder deformations of the standard Gaussian or compactly supported within the unit ball with prescribed decay near the boundary. Under the assumption that the data distribution is \( \beta \)-Hölder (possibly with \( \beta > 1 \)), they derive estimates in the \( \mathcal{H}^{\beta+1} \) norm.

In contrast, our approach operates under weaker assumptions. We extend the unit-ball-supported case to distributions supported on arbitrary convex compact sets, and in the full-support setting, we allow for a wider class of decay at infinity. In Theorem \ref{theo:highregu}, we establish regularity estimates on the score function in all \( \mathcal{H}^\gamma \) norms, for any \( \gamma \geq 0 \). Furthermore, in Corollary \ref{coro:diffusregutime} we derive bounds on the time regularity of the score function, thereby capturing a broader and more flexible notion of regularity in both space and time.

Moreover, the adaptivity of the score to the data regularity enables notably simpler proofs of its efficient approximability by neural networks. In particular we generalize Lemma B.5 and Lemma 3.6 of \cite{oko2023diffusion} in Corollaries~\ref{coro:approxnn} and~\ref{coro:approxnn2}.
 
\subsubsection*{Organisation of the paper}
The remainder of the paper is structured as follows.

\begin{itemize}
    \item In Section \ref{sec:preliandnot}, we introduce the main notations and definitions used throughout the paper.
    
    \item In Section \ref{sec:mainresults}, we present our main results on the Lipschitz-type regularities. We begin in Section~\ref{sec:onesidedL} with bounds on the largest eigenvalue of the Jacobian of the score, and in Section~\ref{sec:Lr} we establish full Lipschitz regularity under additional assumptions. Section~\ref{sec:ltm} then discusses implications for the Lipschitz continuity of the transport map constructed via the probability flow ODE.
    
    \item In Section \ref{sec:hiordersec} we present our main result on the higher regularity of the score. In Sections~\ref{sec:hor}, we prove higher-order regularity estimates for both space and time in Hölder spaces. Then, in Section \ref{sec:appnn} we apply these results to derive efficient neural network approximations of the score function.
    
    \item Section~\ref{sec:theo:mainlipreguofscore} is devoted to the proof of Theorem~\ref{theo:mainlipreguofscore}, which establishes our main result on one-sided Lipschitz regularity of the score function. The detailed proofs of the auxiliary results used in this section, along with the derivation of the remaining Lipschitz estimates, are provided in Appendix~\ref{sec:detailpt1}.

    \item Appendix \ref{sec:theo:highregu} contains the proof of Theorem~\ref{theo:highregu} and Corollary \ref{coro:diffusregutime} on higher-order regularity. In particular, we derive space-time Hölder bounds on the score and prove the approximation results with neural networks.
\end{itemize}

\subsection{Preliminaries and notations}\label{sec:preliandnot}
A key property of the forward process \( \overrightarrow{X}_t \), defined as the solution to equation~\eqref{eq:generalforward}, is that its law admits an explicit expression. Specifically, if \( X_0 \sim p \) and \( Z \sim \mathcal{N}(0, I_d) \) are independent, then the marginal distribution of \( \overrightarrow{X}_t \) satisfies
\begin{align}\label{align:lawofxt}
    \overrightarrow{X}_t \sim e^{-\lambda u_t}X_0 + \sigma_t Z
    ,
\end{align}
where 
\begin{align}
\label{eq:sigma_t}u_t := \int_0^t \gamma_s \mathrm{d} s \quad \text{ and } \quad 
 \sigma_t^2
 := 
\frac{\sigma^2}{\lambda}
\left(
1 - e^{-2\lambda u_t}
\right)
.
\end{align}
To lighten the notation and analysis, we will consider $p_t$ as the law of $X_t$  being the solution to the standard  Ornstein-Uhlenbeck process
\begin{align}
\label{eq:forward-true}
\mathrm{d} X_t
=
-  X_t \mathrm{d} t 
+
\sqrt{2} \mathrm{d} B_t,\quad
X_0 \sim p(x) \mathrm{d} x
,
\end{align}
corresponding to equation \eqref{eq:generalforward} in the case $\lambda=\gamma_t=\sigma=1$. Now as equation \eqref{eq:generalforward} is just a time-reparameterization and space-dilatation of equation \eqref{eq:forward-true}, the regularity results on the score function 
\begin{equation}\label{eq:scfunc}
    s(t,x):= \nabla \log p_t(x),
\end{equation}
with $p_t$ the law of $X_t$, directly transfer (by change of variable) to regularity results on $s^\star(t,x)= \nabla \log \overrightarrow{p}_t(x)$ with $\overrightarrow{p}_t$ the law of $\overrightarrow{X}_t$.

In this paper, we are interested in the Lipschitz regularity of the score function \( s(t, \cdot) \), and more specifically, in the time-integrability of the largest eigenvalue and the operator norm of its Jacobian. Throughout, we denote by \( \|x\| \) the Euclidean norm of a vector \( x \in \mathbb{R}^d \), and by \( \mathbb{S}^{d-1} := \{ x \in \mathbb{R}^d \mid \|x\| = 1 \} \) the unit sphere in \( \mathbb{R}^d \). For a linear map \( L : \mathbb{R}^d \to \mathbb{R}^d \), we define its largest and smallest eigenvalues as
\[
\lambda_{\max}(L) := \sup_{x \in \mathbb{S}^{d-1}} \langle x, Lx \rangle, \quad \lambda_{\min}(L) := \inf_{x \in \mathbb{S}^{d-1}} \langle x, Lx \rangle,
\]
and its operator norm as
\[
\|L\| := \sup_{x \in \mathbb{S}^{d-1}} |\langle x, Lx \rangle|=|\lambda_{\max}(L)|\vee |\lambda_{\min}(L)|.
\]
These quantities will be used to characterize the regularity of the score function, particularly in terms of one-sided and two-sided Lipschitz estimates. Throughout, we use $\nabla^k$ and $\partial_t^k$ to denote the $k$-th order differential operators with respect to the spatial variable $x \in \mathbb{R}^d$ and the time variable $t \in \mathbb{R}_+$, respectively.

We now introduce the Hölder spaces, which provide a natural framework for describing higher-order regularity.
For $\eta \geq 0$, define
$
\lfloor \eta \rfloor := \max\{k \in \mathbb{N}_0 \mid k \leq \eta\}$. 
Let $\mathcal{X} \subset \mathbb{R}^d$, $\mathcal{Y} \subset \mathbb{R}^p$ be subsets of Euclidean spaces,
and let $(f_1, \ldots, f_p) =: f : \mathcal{X} \to \mathcal{Y}$ belong to $C^{\lfloor \eta \rfloor}(\mathcal{X}, \mathcal{Y})$, meaning that  it is $\lfloor \eta \rfloor$-times continuously differentiable.
For any multi-index $\nu = (\nu_1, \ldots, \nu_d) \in \mathbb{N}_0^d$ with $|\nu| := \nu_1 + \cdots + \nu_d \leq \lfloor \eta \rfloor$, denote the corresponding partial derivative of each component $f_j$ by
$
\partial^\nu f_j := \frac{\partial^{|\nu|} f_j}{\partial x_1^{\nu_1} \cdots \partial x_d^{\nu_d}}.
$
Writing
$
\|f_j\|_{\mathcal{X},\alpha} := \sup_{\substack{x ,y\in \mathcal{X}\\x\neq y}} \frac{|f_j(x) - f_j(y)|}{\|x - y\|^{\alpha}}
$
,
we define
\begin{align*}
    \| f \|_{\mathcal{H}^\eta(\mathcal{X})}
    :=
    \max_{1 \leq j \leq d} \Bigl\{ \sum_{0\leq |\nu| \leq \lfloor \eta \rfloor} \|\partial^\nu f_j\|_{L^\infty(\mathcal{X})} + \sum_{|\nu| = \lfloor \eta \rfloor} \|\partial^\nu f_j\|_{\mathcal{X},\eta - \lfloor \eta \rfloor} \Bigr\}.
\end{align*}

Then, for $K > 0$, define the Hölder ball of regularity $\eta$ and radius $K$ by
\begin{align}\label{align:holderspace}
\mathcal{H}^\eta_K(\mathcal{X}, \mathcal{Y}) := \left\{ f \in C^{\lfloor \eta \rfloor}(\mathcal{X}, \mathcal{Y}) \ \middle| \ \| f \|_{\mathcal{H}^\eta(\mathcal{X})} \leq K \right\}.
\end{align}

In the following, we will often take $\mathcal{X}$ being $\mathbb{R}^d$ or $supp(p)$ the support of the distribution $p$. In the case where $\mathcal{X}=\mathbb{R}^d$, we write $\mathcal{H}^\eta$ instead of $\mathcal{H}^\eta(\mathbb{R}^d)$ for lightness of notation. For $p>0$, we make use of the Wasserstein distance \citep{villani2009optimal} which is defined as
\begin{equation}
W_p(\mu, \nu) = \left( \inf_{\pi \in \Pi(\mu, \nu)} \int_{\mathbb{R}^d \times \mathbb{R}^d} \|x - y\|^p d\pi(x, y) \right)^{1/p},
\end{equation}
where $\Pi(\mu, \nu)$ is the set of all couplings between the probability measures $\mu$ and $\nu$.

\section{Lipschitz-type regularity}\label{sec:mainresults}
In this section, we present and discuss the main results of the paper on the Lipchitz-regularity of the score function. We begin by introducing the assumptions placed on the data distribution 
$p$, which are used to establish Lipschitz estimates for the score function.

\begin{assumption}\label{assum:1}
The probability density $p:\mathbb{R}^d\rightarrow \mathbb{R}$ is of the form
$p(x)=\exp(-u(x)+a(x))$ and satisfies both the following assumptions
\begin{enumerate}
    \item $u\in C^2$ on the interior of the support of $p$ and there exists $\alpha>0$ such that   $\nabla^2 u \succeq \alpha\text{Id}$, 
    \item there exists $\beta\in (0,1]$ and $K>0$ such that for all $x,y\in \mathbb{R}^d$ with $\|x-y\|\leq 1$, we have\\ $|a(x)-a(y)|\leq K\|x-y\|^\beta$.
\end{enumerate}
\end{assumption}
Assumption~\ref{assum:1} ensures that the data distribution admits a density with respect to the Lebesgue measure that is both sub-Gaussian and Hölder regular. Importantly, this does not imply that \( p \) is log-concave, since the function \( x \mapsto u(x) - a(x) \) does not need to be convex. In addition to Assumption~\ref{assum:1}, we impose further structural conditions to control the behavior of the distribution.

\begin{assumption}\label{assum:2}
The  probability density $p:\mathbb{R}^d\rightarrow \mathbb{R}$ satisfies Assumption \ref{assum:1} and \textbf{one} of the following three conditions holds:
\begin{enumerate}
    \item The support of $p$ is compact.
    \item $\nabla^2 u \preceq A\text{Id}$, for some $A>0$.
    \item $\|a\|_\infty\leq K$ for some $K>0$.
\end{enumerate}
\end{assumption}
Assumption~\ref{assum:2} offers three alternative structural conditions on the data distribution: either compact support, a uniform upper bound on the curvature of the potential $u$, or a bound on the supremum norm of the perturbation $a$. This assumption is used to control the integral  $\int \exp\left(a(y)-a(\int z d\mu(z))\right)d\mu(y)$, for certain key probability measures $\mu$ related to $p$ (see Lemma \ref{lemma:boundonratios} for details).

Throughout the paper, we use \( C, C_2 > 0 \) to denote generic constants that may vary from line to line. These constants depend only on the parameters \( \alpha, \beta, K, d \), and either on the diameter of the support of \( p \) under Assumption~\ref{assum:2}-1 or on the curvature bound \( A \) under Assumption~\ref{assum:2}-2. No dependence on other quantities (e.g., the point \( x \), time \( t \), or regularity exponents \( \gamma \)) is introduced unless explicitly stated.

We do not attempt to optimize the dependence of these constants on the ambient dimension \( d \). While we believe that the results could in principle be made dimension-free, such a refinement would require significantly more technical effort, and we leave this direction for future work. The goal of this paper is to provide sharp (with respect to time) and flexible regularity estimates under general assumptions, in a framework broad enough to accommodate practical scenarios and ensure the stability of learned generative processes.

\subsection{One-sided Lipschitz regularity}\label{sec:onesidedL}
We start by studying the one-sided Lipschitz regularity of the score function, that is, a bound on the largest eigenvalue of its Jacobian.
In the context of score-based generative modeling, the time-integrability of the largest eigenvalue \( \lambda_{\max}(\nabla s(t,x)) \) is more relevant than that of the full operator norm \( \|\nabla s(t,x)\| \). This is because stability analyses based on Grönwall-type arguments typically involve only the maximal eigenvalue of the Jacobian, as illustrated by the following result.

\begin{proposition}
\label{prop:SDEstability-drift}
Let \( t > 0 \), and let \( (X_s)_{s \in [0,t]} \) and \( (\bar{X}_s)_{s \in [0,t]} \) be the solutions to the SDEs
\begin{align*}
 \mathrm{d}X_s &= a_s(X_s) \, \mathrm{d}s + b_s \, \mathrm{d}B_s, \quad X_0 \sim \mu_0, \\
 \mathrm{d}\bar{X}_s &= \bar{a}_s(\bar{X}_s) \, \mathrm{d}s + b_s \, \mathrm{d}B_s, \quad \bar{X}_0 \sim \mu_0,
\end{align*}
where \( b_s \geq 0 \), and \( (a_s)_s \), \( (\bar{a}_s)_s \) are locally space-Lipschitz vector fields. Denoting by \( \mu_s = \mathrm{Law}(X_s) \), \( \bar{\mu}_s = \mathrm{Law}(\bar{X}_s) \) and $L_s= \sup_{x \in \mathbb{R}^d} \lambda_{\max}(\nabla a_s(x))$, we have
\begin{align*}
    W_2^2(\mu_t, \bar{\mu}_t)
    \leq t
    \int_0^t \int_{\mathbb{R}^d}
    e^{2\int_s^t L_u \mathrm{d}u} 
    \|a_s(x) - \bar{a}_s(x)\|^2 \, \mathrm{d}\bar{\mu}_s(x) \, \mathrm{d}s.
\end{align*}
\end{proposition}

The proof of Proposition~\ref{prop:SDEstability-drift} can be found in Section~\ref{sec:prop:SDEstability-drift}. This result shows that if the largest eigenvalue of the Jacobian of one of the drifts is integrable over time, then the Wasserstein distance between the laws of the two SDE solutions is proportional to the discrepancy between their drifts. In the context of score-based generative modeling, this is a key tool for analyzing model stability and convergence, as it directly relates the model error to the accuracy of the score approximation. When the initial distributions differ, a similar bound still holds (see for instance Proposition 4.3 in \cite{pedrotti2023improved}). We now state the main result of the paper, which provides an upper bound on the largest eigenvalue of the Jacobian of the score function.
\begin{theorem}[One-sided-Lipschitz regularity]\label{theo:mainlipreguofscore}
Let $p:\mathbb{R}^d\rightarrow \mathbb{R}$ a  probability density satisfying Assumption \ref{assum:1}. If it additionally satisfies either Assumption \ref{assum:2}-1, \ref{assum:2}-2 or  \ref{assum:2}-3, then  
for all $t>0$ the score function $s$ \eqref{eq:scfunc} associated to equation \eqref{eq:forward-true} satisfies 
       \begin{equation*}
\sup_{x\in \mathbb{R}^d} \lambda_{\max}\big(\nabla s(t,x)\big)\leq Ce^{-2t}(1+t^{\frac{\beta}{2}-1})-1.
\end{equation*}
\end{theorem}
The proof of Theorem \ref{theo:mainlipreguofscore} can be found in Section  \ref{sec:theo:mainlipreguofscore}.
This result shows that the largest eigenvalue of the Jacobian of the score function is uniformly bounded over space, with an explicit decay rate as a function of time. The presence of the term \( t^{\frac{\beta}{2}-1} \) reflects the influence of the Hölder regularity of the perturbation \( a \), and the exponential prefactor ensures that the dominant behavior is governed by the underlying Gaussian dynamics. Such a bound is crucial for understanding the stability and convergence of generative methods that rely on the score function, especially in the presence of non-log-concave perturbations.
In particular, from the integrability near $0$ of the function $t\mapsto t^{\frac{\beta}{2}-1}$, we deduce the following corollary.
\begin{corollary}\label{coro:finiteintegrabilityofscore}
 Let $p:\mathbb{R}^d\rightarrow \mathbb{R}$ a  probability density satisfying Assumption \ref{assum:1}. If it additionally satisfies either Assumption \ref{assum:2}-1, \ref{assum:2}-2 or  \ref{assum:2}-3, then  
for all $t>0$ the score function $s$ \eqref{eq:scfunc} associated to equation \eqref{eq:forward-true} satisfies 
       \begin{equation*}
\int_0^\infty \sup_{x\in \mathbb{R}^d} \lambda_{\max}\big(\nabla s(t,x)+\text{Id}\big)\leq C.
\end{equation*}
\end{corollary}
This corollary shows that the score satisfies a uniform one-sided-Lipschitz integrability condition over time, which is a key technical assumption in the convergence analysis of score-based generative models as illustrated by Proposition \ref{prop:SDEstability-drift}. In particular, it allows one to control the growth of trajectories such as those arising in reverse-time ODEs and SDEs \citep{lions2024transport}. In this sense, Corollary~\ref{coro:finiteintegrabilityofscore} provides a theoretical justification for the one-sided Lipschitz assumptions employed in works such as \cite{kwon2022score,pedrotti2023improved}.

\subsection{Lipschitz regularity}\label{sec:Lr}
We now turn to the question of whether the score function can be fully Lipschitz, that is, whether the lower bound on the smallest eigenvalue of the Jacobian matches the upper bound established in Theorem~\ref{theo:mainlipreguofscore}. In particular, we are interested in whether the Jacobian is uniformly bounded in operator norm. As a preliminary remark, Proposition~\ref{prop:values} yields, for all \( t > 0 \) and \( x \in \mathbb{R}^d \),
\begin{equation}\label{align:naivelowerbound}
\lambda_{\min}(\nabla s(t,x)) \geq -C e^{-2t}(1 + t^{-1}) - 1.
\end{equation}
However, this lower bound does not suffice to guarantee the finiteness of the integral
$
\int_0^{\infty} \|\nabla s(t, \cdot)\|_{\infty} \, dt.$ To study the regularity of the operator norm more precisely, we distinguish between the compactly supported case and the case where $p$ has full support.

\subsubsection{The compact case}\label{sec:lipregucompact}
In the case where the data distribution 
$p$ is compactly supported, the lower bound in equation~\eqref{align:naivelowerbound} turns out to be essentially sharp and cannot be improved, as formalized in the following result.

\begin{proposition}\label{prop:unboundjaccompact}
Let $p:\mathbb{R}^d\rightarrow \mathbb{R}$ a probability density that is compactly supported. Then  
for all $t>0$, the score function $s$ \eqref{eq:scfunc} associated to equation \eqref{eq:forward-true} satisfies 
  \begin{equation*}
\sup_{x\in \mathbb{R}^d} \|\nabla s(t,x)\|\geq C^{-1}e^{-2t}(1+t^{-1})+1,
\end{equation*}
\end{proposition}

The proof of Proposition \ref{prop:unboundjaccompact} can be found in Section \ref{sec:prop:unboundjaccompact}. This result highlights a fundamental limitation in the compact case: the growth of the operator norm as 
$t\to 0$ cannot be mitigated, implying that the Jacobian is not uniformly integrable over time.

\subsubsection{The full support case} 
When the probability density \( p \) has full support and satisfies Assumption~\ref{assum:2}-2, the operator norm of the Jacobian of the score can be bounded by the same quantity as the upper bound on \( \lambda_{\max} \) given in Theorem~\ref{theo:mainlipreguofscore}.

\begin{corollary}[Lipschitz regularity]\label{coro:lipschitzbound}
 Let $p:\mathbb{R}^d\rightarrow \mathbb{R}$ a  probability density satisfying Assumptions \ref{assum:1} and \ref{assum:2}-2. Then  
for all $t>0$, the score function $s$ \eqref{eq:scfunc} associated to equation \eqref{eq:forward-true} satisfies 
       \begin{equation*}
\sup_{x\in \mathbb{R}^d} \|\nabla s(t,x)\|\leq Ce^{-2t}(1+t^{\frac{\beta}{2}-1})-1.
\end{equation*}
\end{corollary}
The proof of Corollary \ref{coro:lipschitzbound} can be found in Section \ref{sec:coro:lipschitzbound}. This result shows that, in contrast to the compact support case, full support combined with bounded curvature ensures global Lipschitz control of the score function over time. In this sense, Corollary~\ref{coro:finiteintegrabilityofscore} provides a theoretical justification for the Lipschitz assumptions employed in works such as \cite{lee2023convergence, chen2022sampling,pmlr-v202-chen23q,wibisono2024optimal} in the case of diffusion, and \cite{chen2023probability,benton2023error,huang2024convergence} in the case of ODE.

\subsection{Lipschitz regularity of the Probability Flow ODE}\label{sec:ltm}
A direct consequence of Theorem~\ref{theo:mainlipreguofscore} is that the transport map constructed by the probability flow ODE \citep{song2020score} is Lipschitz. Indeed, setting the noise level \( b_t = 0 \) in the backward process \eqref{eq:truebackward}, the solution at final time \( X_T \) defines a deterministic transport map from \( p_T \) to \( p \). In particular, taking the limit as \( T \to \infty \), this construction yields a transport map from the standard Gaussian to the target distribution \( p \).

This limiting map is known as the \emph{Kim–Milman transport map} \citep{kim2012generalization}, also referred to as the \emph{Langevin transport map} \citep{fathi2024transportation}, and has been the subject of extensive recent study; see for instance \citep{mikulincer2023lipschitz,lopez2024bakry,conforti2025coupling}. The one-sided Lipschitz regularity established in Theorem~\ref{theo:mainlipreguofscore} ensures the time-integrability of the velocity field associated with the transport map, which directly implies its global Lipschitz continuity.

\begin{proposition}[The Probability flow ODE is Lipschitz]\label{prop:liptran} Let $p:\mathbb{R}^d\rightarrow \mathbb{R}$ a  probability density satisfying Assumption \ref{assum:1} and  additionally satisfying either Assumption \ref{assum:2}-1, \ref{assum:2}-2 or  \ref{assum:2}-3.  Then, the transport map from the $d$-dimensional Gaussian distribution to $p$ built by the probability flow ODE is Lipschitz.
\end{proposition}

The proof of Proposition \ref{prop:liptran} can be found in Section \ref{sec:prop:liptran}. Interestingly, the one sided-Lipschit regularity of the score is enough to provide the complete Lipschitz regularity of the map. Compared to previous works on the Kim-Milman transport map, our main contribution lies in relaxing the typical Lipschitz regularity assumption on the data distribution to a weaker Hölder condition. However, in the case of Lipschitz perturbations, Assumption~\ref{assum:2} is not required to guarantee the existence of a Lipschitz transport map. This raises the open question of what are the minimal conditions required in the Hölder-perturbed setting to ensure such regularity.

\section{Higher-order regularity and neural network approximation}\label{sec:hiordersec}

\subsection{Higher order regularity}\label{sec:hor}
We now go beyond Lipschitz regularity and establish estimates on higher-order derivatives of the score function. This type of regularity plays an important role in the analysis of minimax rates \citep{Tsybakov08,oko2023diffusion} which aims to describe how the statistical convergence rates scale with the regularity of the data. To this end, we introduce a refined set of assumptions that allows us to control the behavior of the score in Hölder spaces of arbitrary order.

\begin{assumption}\label{assum:higherorder}
The probability density $p:\mathbb{R}^d\rightarrow \mathbb{R}$ satisfies Assumption \ref{assum:1}-1 and there exists $\beta>0$, $K\geq 1$ such that $p$ satisfies all the following assumptions:
\begin{enumerate}
    \item $a$ belongs to the Hölder space  $ \mathcal{H}^{\beta}_K(supp(p))$ defined in \eqref{align:holderspace}.
    \item  $u$ satisfies $\|\argmax_x e^{-u(x)}\|\leq K$. 
    \item $u$ satisfies $u\in C^{\lfloor \beta \rfloor+1}$ on the interior of the support of $p$ and   for all $x,y \in \mathbb{R}^d$ and $i\in \{0,...,\lfloor \beta \rfloor+1\}$, we have $\|x-y\|\leq (K|u(x)|\vee 1)^{-K} \Rightarrow \|\nabla^i u(y)\|\leq K(|u(x)|^K+1).$
\end{enumerate}
\end{assumption}

Assumption~\ref{assum:higherorder}-1 ensures that the perturbation \( a \) is of smoothness \( \beta \), naturally extending Assumption~\ref{assum:1}-2 to higher-order regularity settings. Assumption~\ref{assum:higherorder}-2 guarantees that the bulk of the probability mass of \( p \) is concentrated near the origin, which in turn ensures that under the forward process~\eqref{eq:forward-true}, the transported mass does not need to travel long distances—an important feature when controlling the time derivatives of the score. 

Assumption~\ref{assum:higherorder}-3 imposes a local regularity condition on the potential \( u \), allowing its derivatives to grow moderately with respect to its value. Although somewhat unconventional, this assumption is in fact very mild and designed to accommodate both compactly supported and full-support distributions. 
For instance, in the compact case, this framework allows potentials of the form \( u(x) = d(x, \mathcal{Y}^c)^{-R} \), with $R>0$ where \( d(\cdot, \mathcal{Y}^c) \) denotes the distance to the complement of a convex compact set \( \mathcal{Y} \). In the full-support setting, the assumption covers a broad class of \( \beta \)-smooth densities, as \( u \) can, for instance, be any convex polynomial. Moreover, Assumption~\ref{assum:higherorder} also accommodates Gaussian mixtures whose modes lie within a prescribed ball: in this case, the perturbation \( a \) can absorb the non-convex behavior of \( p \) inside the ball, while the convexity of \( u \) governs the Gaussian behavior outside.

\subsubsection{Space-regularity}
Under the regularity conditions of Assumption \ref{assum:higherorder}, we show that the score function becomes progressively smoother under the action of the forward process \eqref{eq:forward-true}. The regularity gain is quantified via Hölder norms over high-probability subsets of \( \mathbb{R}^d \) (with respect to $p_t$ the law of the Ornstein-Uhlenbeck process), and adapts to the smoothness of the initial distribution.

\begin{theorem}\label{theo:highregu}
Let $p:\mathbb{R}^d\rightarrow \mathbb{R}$ a  probability density satisfying Assumption \ref{assum:1}-1 and Assumption \ref{assum:higherorder}. Then for all $\epsilon\in (0,1/4)$ and $t>0$, there exists a convex set $A_t^\epsilon \subset B(0,C\log(\epsilon^{-1})^{C_2})$ such that $p_t(A_t^\epsilon)\geq 1-\epsilon$ and for all $\gamma\geq 0$, we have
$$ \|s(t,\cdot)+\text{Id}(\cdot)\|_{\mathcal{H}^{\gamma}(A_t^\epsilon)}\leq C_\gamma\log(\epsilon^{-1})^{C_2(1 + \gamma)}e^{-t} \left(1 + t^{-\frac{1}{2} \big( (1 + \gamma - \beta) \vee 0 \big)} \right).$$
\end{theorem}

The proof of Theorem \ref{theo:highregu} can be found in Section \ref{sec:theo:highregu}. This result shows that the score function becomes increasingly regular over time, and that its smoothness adapts to the regularity \( \beta \) of the underlying data distribution. In particular, when \( 1+\gamma \leq \beta \), the bound exhibits no singularity as \( t \to 0 \), while for \( 1+\gamma > \beta \), the loss in regularity is precisely quantified by a power of \( t \). The appearance of the multiplicative factor \( \log(\epsilon^{-1}) \) captures the trade-off between spatial localization and regularity: the estimate holds on a high-probability subset of \( \mathbb{R}^d \) whose measure is at least \( 1 - \epsilon \). As discussed in Section~\ref{sec:lipregucompact}, it is generally impossible to obtain uniform higher-order regularity estimates over the entire space \( \mathbb{R}^d \), particularly when the data distribution is compactly supported. The localization to \( A_t^\epsilon \) is therefore both natural and necessary.

Importantly, the set \( A_t^\epsilon \) changes only once with respect to time. Indeed, there exists a universal constant \( C^\star > 0 \), independent of \( \epsilon \), such that:
\begin{itemize}
    \item for \( t >  \log(\epsilon^{-1})^{-C^\star} \), we have $A_t^\epsilon = A_\infty^\epsilon$ with
    \begin{equation}\label{eq:dkdkoididjudueueueieoz}
     A_\infty^\epsilon := \left\{ y \in \mathbb{R}^d \,\middle|\, \|y - y^\star\|^2 \leq C^\star \left( \log(\epsilon^{-1})(1 + \alpha^{-1}) + K \right) \right\},
    \end{equation}
    \item for \( t \leq \log(\epsilon^{-1})^{-C^\star} \), we have $A_t^\epsilon = A_0^\epsilon$ with
    \begin{equation}\label{eq:hgfjkzudijdjhdhdhhdhd}
    A_0^\epsilon := \left\{ y \in \mathbb{R}^d \,\middle|\, u(y) \leq K \left( \log(\epsilon^{-1}) + R + 1 \right)^K \right\}.
    \end{equation}
\end{itemize}

The intuition behind this definition is the following: when \( t \) is small, the law \( p_t \) of the Ornstein--Uhlenbeck process remains concentrated near the support of the original distribution \( p \). Hence, \( A_t^\epsilon \) corresponds to a high-probability set for \( p \). On the other hand, for \( t \) larger than \( \log(\epsilon^{-1}) \), the score function becomes sufficiently regular, with its potential singularity controlled by a logarithmic factor. In this regime, \( A_t^\epsilon \) can be chosen as a high-probability set for both \( p \) and the Gaussian measure \( \gamma_d \).

The key advantage of this piecewise definition is that \( A_t^\epsilon \) does not vary continuously with time, but rather switches only once. This greatly simplifies the construction of time-space approximations of the score function by neural networks presented in Section~\ref{sec:appnn}.

\subsubsection{Joint space-time regularity}
The regularity of the score function in space also enables control over its regularity in time. This connection is formalized using the evolution identity
\begin{align}\label{eq:eqdtQ}
    \partial_t \left( \frac{p_t}{\gamma_d} \right)(x) = \Delta \left( \frac{p_t}{\gamma_d} \right) (x) - \langle x , \nabla \left( \frac{p_t}{\gamma_d} \right) (x) \rangle,
\end{align}
where \( \gamma_d \) is the standard \( d \)-dimensional Gaussian density. This identity, derived from the Fokker–Planck equation (see Section~2.1 of \cite{mikulincer2023lipschitz}), highlights how the time derivative of the density depends on its spatial derivatives. Consequently, high-order space regularity yields control over time regularity as well.

\begin{corollary}\label{coro:diffusregutime}
Let $p:\mathbb{R}^d\rightarrow \mathbb{R}$ a  probability density satisfying Assumption \ref{assum:1}-1 and Assumption \ref{assum:higherorder}. Then for all $\epsilon\in (0,1/4)$ and $t>0$, there exists a convex set $A_t^\epsilon \subset B(0,C\log(\epsilon^{-1})^{C_2})$ such that $p_t(A_t^\epsilon)\geq 1-\epsilon$ and for all $\gamma\geq 0$,  $x\in A_t^\epsilon$ we have
\begin{equation}\label{eq:reguscoretime}
\|s(\cdot,x)+x\|_{ \mathcal{H}^{\gamma}([t,\infty))}\leq C_\gamma\log(\epsilon^{-1})^{C_2(1+\gamma)}e^{-t}\left(1+t^{-\big((\frac{1}{2}+\gamma-\frac{\beta}{2})\vee 0\big)}\right)
\end{equation}
and
for all $k\in \mathbb{N}_{\geq 0}$, 
\begin{equation}\label{eq:reguscorespace}\|\partial_t^k \big(s(t,\cdot)+\text{Id}(\cdot)\big)\|_{\mathcal{H}^{\gamma}(A_t^\epsilon)}\leq C_{k,\gamma}\log(\epsilon^{-1})^{C_2(1+k+\gamma)}e^{-t}\left(1+t^{-\big((\frac{1}{2}+k+\frac{\gamma-\beta}{2})\vee 0\big)}\right).
\end{equation}
\end{corollary}

The proof of Corollary \ref{coro:diffusregutime} can be found in Section \ref{sec:timeregu}. Under the same assumptions as Theorem~\ref{theo:highregu}, this result provides estimates on the time regularity of the score function \( s(t,x) \) both as a function of time at fixed spatial points, and in terms of time derivatives as functions over space. In particular, it shows that the  \( \mathcal{H}^\gamma \) time-regularity scales as $t^{-\gamma}$ in contrast to the space-regularity which scales at $t^{-\gamma/2}$. This discrepancy arises from equation \eqref{eq:eqdtQ} which shows that time-derivatives of the score involve second order space-derivatives.

\subsection{Application to score approximation by neural networks}\label{sec:appnn}
Corollary~\ref{coro:diffusregutime} is particularly useful for analyzing statistical optimality in score-based generative models. For example, the recent works of \cite{oko2023diffusion} and \cite{fukumizu2024flow} demonstrate that diffusion and flow-matching models can achieve minimax convergence rates when the target density lies in a Besov space. A key step in their proofs is to construct a neural network that approximates the score function with controlled complexity. Specifically, Lemma B.5 in \cite{oko2023diffusion} shows that for $p\in \mathcal{H}^{\beta}_K([0,1]^d)$ bounded below, there exists a ReLU neural network $N^\epsilon$, with $O(\epsilon^{-d})$ parameters such that for all $t\in [\epsilon^{2\beta+1},\epsilon^2]$
\begin{align}\label{align:nnaprox}
\int \|N^\epsilon(t,x)-s(t,x)\|^2dp_t(x)\leq C\log(\epsilon^{-1}) \frac{\epsilon^{2\beta}}{t}.
\end{align}

However, their proof relies on extended technical constructions that do not explicitly leverage the adaptive regularity of the score function. In contrast, Corollary~\ref{coro:diffusregutime} directly quantifies how the regularity of the score adapts to the smoothness of the data. In particular, using Theorem \ref{theo:highregu} with $\gamma=\beta$ and the classical approximation theory of neural networks \citep{schmidt2020nonparametric,adaptativity,de2021approximation} we directly obtain \eqref{align:nnaprox} for a fixed time $t$ in our setting. Now, using the time regularity of the score given by Corollary~\ref{coro:diffusregutime}, we can recover \eqref{align:nnaprox} for all $t\in [\epsilon^{2\beta+1},\epsilon^2]$.
\begin{corollary}\label{coro:approxnn}
 For all $\epsilon\in (0,1/4)$ and $\theta\geq 2$, there exists a neural network class $\mathcal{F}$ (either ReLU or tanh) such that its covering number satisfies for all $\eta\in (0,\epsilon)$,
$$\log (\mathcal{N}(\mathcal{F},\|\cdot\|_\infty,\eta))\leq C_\theta \log(\eta^{-1})^2\epsilon^{-d}$$
and for any probability density $p:\mathbb{R}^d\rightarrow \mathbb{R}$ satisfying Assumption \ref{assum:1}-1 and Assumption \ref{assum:higherorder}, there exists $s_\epsilon\in \mathcal{F}$ such that for all $t\in [\epsilon^\theta,\epsilon^2]$ we have
$$\int \|s_\epsilon(t,x)-s(t,x)\|^2dp_t(x)\leq  C_\theta\log(\epsilon^{-1})^{C_2}\frac{\epsilon^{2\beta}}{t},$$
with $s$ the score function associated to $p$.
\end{corollary}

The proof of Corollary~\ref{coro:approxnn} can be found in Section~\ref{sec:coro:approxnn}. This result highlights how the adaptive regularity of the score function enables efficient neural network approximations when \( t \leq \epsilon^2 \), where the score may exhibit limited smoothness. In contrast, when \( t \geq \epsilon^2 \), the score function becomes smoother and can be more effectively approximated using slightly larger networks. Specifically, Lemma 3.6 in~\cite{oko2023diffusion} shows that if \( p \in \mathcal{H}^{\beta}_K([0,1]^d) \) is bounded below, then there exists a ReLU neural network \( N^\epsilon \) with \( O(\epsilon^{-d - \delta}) \) parameters such that for all \( t \in [\epsilon^2, \log(\epsilon^{-1})] \),
\begin{align}\label{align:nnaprox2}
\int \|N^\epsilon(t,x)-s(t,x)\|^2dp_t(x)\leq C_\delta  \frac{\epsilon^{2\beta+1}}{t}.
\end{align}
Similarly to Corollary \ref{coro:approxnn}, the following result recovers inequality \eqref{align:nnaprox2} within our framework with a significantly simpler proof, further illustrating the benefits of our regularity analysis.

\begin{corollary}\label{coro:approxnn2}
 For all $\epsilon\in (0,1/4)$ and $\delta>0$, there exists a neural network class $\mathcal{F}$ (either ReLU or tanh) such that its covering number satisfies for all $\eta\in (0,\epsilon)$,
$$\log (\mathcal{N}(\mathcal{F},\|\cdot\|_\infty,\eta))\leq C_\delta \log(\eta^{-1})^2\epsilon^{-d-\delta}$$
and for any probability density $p:\mathbb{R}^d\rightarrow \mathbb{R}$ satisfying Assumption \ref{assum:1}-1 and Assumption \ref{assum:higherorder}, there exists $s_\epsilon\in \mathcal{F}$ such that for all $t\in [\epsilon^2,\log(\epsilon^{-1})]$ we have
$$\int \|s_\epsilon(t,x)-s(t,x)\|^2dp_t(x)\leq  C_\delta\log(\epsilon^{-1})^{C_2}\frac{\epsilon^{2\beta+1}}{t},$$
with $s$ the score function associated to $p$.
\end{corollary}

The proof of Corollary~\ref{coro:approxnn2} can be found in Section~\ref{sec:coro:approxnn2}. Corollaries~\ref{coro:approxnn} and~\ref{coro:approxnn2} illustrate how the adaptivity of the score to the regularity of the data distribution can simplify the analysis of statistical optimality in score-based generative models, a direction we plan to explore in a future work.

\section{Proof of Theorem \ref{theo:mainlipreguofscore}}\label{sec:theo:mainlipreguofscore}

\subsection{Preliminaries}
 We write $C>0$ for any constant that depends only on $\alpha,\beta,K,d$ and the diameter of the support of $p$ under Assumption \ref{assum:2}-1 or $A$ under Assumption \ref{assum:2}-2. This quantity may vary from lines to lines or even in a single line.
Let us define some key quantities that will be used throughout the proof. Let 
\begin{equation}\label{eq:r}
 r(x):=\frac{p(x)}{\gamma_d(x)}
\end{equation}
the density of $p$ with respect to the standard Gaussian distribution and
\begin{equation}\label{eq:Qtr}
    Q_t r(x):= \int \varphi^{t,x}(y)r(y)dy= \int r(e^{-t}x+\sqrt{1-e^{-2t}}z)d\gamma_d(z),
\end{equation}
with $\varphi^{t,x}$ the density of the $d$-dimensional Gaussian measure with mean $e^{-t}x$ and covariance $(1-e^{-2t})\text{Id}$. In particular, we have
$p_t(x)=\gamma_d(x)Q_tr(x)$. The quantity \( Q_t \) plays a central role in the analysis of score regularity, as the score is given by
\begin{equation}\label{eq:jnhodishsiozkdp}
\nabla \log p_t(x) = -x + \frac{\nabla Q_t r(x)}{Q_t r(x)}.
\end{equation}
By defining the probability density
\begin{equation}\label{eq:ptx}
    p^{t,x}(y)=\frac{1}{Q_tr(x)}r(y)\varphi^{t,x}(y),
\end{equation}
and the deviation of a point \( y \in \mathbb{R}^d \) from its mean
\begin{equation}\label{eq:noth}
    H_p(y,x) := y - \int z \, dp^{t,x}(z),
\end{equation}
we can express both the score function and its Jacobian.

\begin{proposition}\label{prop:values}
For all $t>0$ and $x\in \mathbb{R}^d$, we have
$$s(t,x)=-x+\frac{e^{-t}}{1-e^{-2t}}\int (y-e^{-t}x)dp^{t,x}(y)$$
and
$$\nabla s(t,x)=-\text{Id}+\frac{e^{-2t}}{(1-e^{-2t})^2}\int H_p(y,x)^{\otimes 2}dp^{t,x}(y)-\frac{e^{-2t}}{1-e^{-2t}}\text{Id}.$$
\end{proposition}

The proof of Proposition \ref{prop:values} can be found in Section \ref{sec:prop:values}.
\noindent
Proposition~\ref{prop:values} suggests that both the score and its Jacobian may blow up as \( t \to 0 \), due to the presence of the singular term \( (1 - e^{-2t})^{-1} \). In fact, the score function can indeed diverge in certain situations. For instance, when \( p \) is compactly supported and \( x \) lies far from the support, the quantity
$
\left\| \int (y - e^{-t}x) \, dp^{t,x}(y) \right\|
$
remains bounded away from zero. As a result, we obtain the estimate \( \|s(t,x)\| = O(t^{-1}) \) as \( t \to 0 \). However, this blow-up does not necessarily occur for the Jacobian of the score. Indeed, the covariance of \( p^{t,x} \), given by
\[
\int H_p(y,x)^{\otimes 2} \, dp^{t,x}(y),
\]
vanishes as \( t \to 0 \), since \( p^{t,x} \) is proportional to \( \varphi^{t,x}(y) r(y) \), and the variance of \( \varphi^{t,x} \) scales with \( 1 - e^{-2t} \). Therefore, the key challenge in obtaining a bound on \( \lambda_{\max}(\nabla s(t,x)) \) is controlling the decay of the covariance of \( p^{t,x} \).

\subsection{Concentration for log-concave measures}\label{sec:conlcm}
To control the covariance of \( p^{t,x} \), we rely on the Brascamp–Lieb inequality, which offers sharp concentration estimates under log-concavity assumptions.

\begin{proposition}\label{prop:BL}[\cite{BRASCAMP1976366}]
    Let a probability density function $q(x)=\exp(-\phi(x))$ with $\phi$ convex. Then, for any derivable function $S:\mathbb{R}^d\rightarrow \mathbb{R}$  we have $$\int \left(S(y)-\int S(z)dq(z)\right)^2dq(y)\leq \int \big(\nabla^2 \phi(y)\big)^{-1}(\nabla S(y) ,\nabla S(y)) dq(y).$$
\end{proposition}
Since \( p \) is of the form \( p(x) = \exp(-u(x) + a(x)) \), with \( a \) only Hölder continuous, the Brascamp--Lieb inequality cannot be applied directly.
 To circumvent this, we introduce an auxiliary measure to which the inequality can be applied.
Let
\begin{equation}\label{eq:q}
q(y) := \exp\left(-u(\|y\|^2) + \frac{\|y\|^2}{2} \right) = r(y) e^{-a(y)},
\end{equation}
which serves as the log-concave counterpart of \( r \). We then define
\begin{equation}\label{eq:nutx}
\nu^{t,x}(y) := \frac{q(y)\varphi^{t,x}(y)}{Q_t q(x)} = \frac{\exp\left(-u(\|y\|^2) + \|y\|^2/2\right)\varphi^{t,x}(y)}{\int \exp\left(-u(\|z\|^2) + \|z\|^2/2\right) \varphi^{t,x}(z) \, dz},
\end{equation}
which plays the role of a log-concave approximation of \( p^{t,x} \). Just as we defined \( H_p(y,x) \), we introduce
\[
H_\nu(y,x) := y - \int z \, d\nu^{t,x}(z),
\]
which measures the deviation of a point \( y \in \mathbb{R}^d \) from the mean under \( \nu^{t,x} \). Using the subadditivity of the maximum eigenvalue, we write
\begin{align*}
\lambda_{\max}\left(\int H_p(y,x)^{\otimes 2} \, dp^{t,x}(y)\right) 
&\leq \lambda_{\max}\left( \int H_p(y,x)^{\otimes 2} \, dp^{t,x}(y) - \int H_\nu(y,x)^{\otimes 2} \, d\nu^{t,x}(y) \right) \\
&\quad + \lambda_{\max}\left( \int H_\nu(y,x)^{\otimes 2} \, d\nu^{t,x}(y) \right).
\end{align*}
Applying Proposition~\ref{prop:values}, we then obtain
\begin{align}\label{align:vfdihsksgs2}
\lambda_{\max}\left( \nabla s(t,x) + \mathrm{Id} \right)
&= \lambda_{\max}\left( \frac{e^{-2t}}{(1 - e^{-2t})^2} \int H_p(y,x)^{\otimes 2} \, dp^{t,x}(y) - \frac{e^{-2t}}{1 - e^{-2t}} \mathrm{Id} \right) \nonumber \\
&\leq \frac{e^{-2t}}{(1 - e^{-2t})^2} \lambda_{\max}\left( \int H_p(y,x)^{\otimes 2} \, dp^{t,x}(y) - \int H_\nu(y,x)^{\otimes 2} \, d\nu^{t,x}(y) \right) \nonumber \\
&\quad + \frac{e^{-2t}}{(1 - e^{-2t})^2} \lambda_{\max}\left( \int H_\nu(y,x)^{\otimes 2} \, d\nu^{t,x}(y) \right) - \frac{e^{-2t}}{1 - e^{-2t}}.
\end{align}
Since the probability measure \( \nu^{t,x} \) is log-concave, we can apply the Brascamp--Lieb inequality to bound the second term in \eqref{align:vfdihsksgs2}.

\begin{lemma}\label{lemma:boundeasy} Let  $\nu^{t,x}$ the  probability measures defined in \eqref{eq:nutx}.  For all $t>0$ and $x\in \mathbb{R}^d$ we have
$$\frac{e^{-2t}}{(1-e^{-2t})^2}\lambda_{\max}\left(\int H_\nu(y,x)^{\otimes 2}d\nu^{t,x}(y)\right)-\frac{e^{-2t}}{1-e^{-2t}}\leq \frac{e^{-2t}(1-\alpha)}{\alpha(1-e^{-2t})+e^{-2t}}.$$
\end{lemma}

The proof of Lemma \ref{lemma:boundeasy} can be found in Section \ref{sec:lemma:boundeasy}. We observe that, thanks to the \( \alpha \)-strong convexity of \( u \) (Assumption~\ref{assum:1}), the second term in equation~\eqref{align:vfdihsksgs2} remains controlled as \( t \to 0 \).

\subsection{Stability of the covariance}
In this section, we focus on the first term of equation~\eqref{align:vfdihsksgs2}, which involves the difference between the covariances of the measures \( p^{t,x} \) and \( \nu^{t,x} \). We begin by providing a preliminary estimate on the operator norm of this difference.

\begin{proposition}\label{prop:firstuglybound} Let  $p^{t,x}$  and $\nu^{t,x}$ be the  probability measures defined in \eqref{eq:ptx} and \eqref{eq:nutx} respectively. Then, there exists $\Theta\subset (\mathbb{N}_{\geq 0}^3)^3$ such  that for all $t>0$, $x\in \mathbb{R}^d$ and $h\in \mathbb{S}^{d-1}$ we have
\begin{align*}
    |h^\top &\left(\int  H_p(y,x)^{\otimes 2}dp^{t,x}(y)-\int H_\nu(y,x)^{\otimes 2}d\nu^{t,x}(y)\right)h|\\
    &\leq C \sum_{\theta\in \Theta}\prod_{i=1}^3\left(\int  |h^\top  H_\nu(y,x)|^{\theta_{i,1}}|\frac{e^{a(y)}Q_tq(x)}{Q_tr(x)}-1|^{\theta_{i,2}}d\nu^{t,x}(y)\right)^{\theta_{i,3}},
\end{align*}
with $\sum_{i=1}^3\theta_{i,1}\theta_{i,3}=2$ and $\bigvee_{i=1}^3\theta_{i,2}\theta_{i,3}= 1$.
\end{proposition}
The proof of Proposition \ref{prop:firstuglybound} can be found in Section \ref{sec:prop:firstuglybound}. We obtain that the spectral norm of the difference between the covariances can be bounded by weighted moments of \( H_\nu \) and powers of the discrepancy term \( \left| \frac{e^{a(y)} Q_t q(x)}{Q_t r(x)} - 1 \right| \). Although the total degree of \( H_\nu \) in these terms is always equal to 2 (reflecting the second-order nature of the covariance), this alone is not sufficient to offset the singular scaling \( (1 - e^{-2t})^{-2} \) that appears in equation~\eqref{align:vfdihsksgs2}. However, we will show that additional concentration can be extracted from the discrepancy term \( \left| \frac{e^{a(y)} Q_t q(x)}{Q_t r(x)} - 1 \right| \), allowing us to effectively control the full expression.
Let us introduce the function
\begin{equation}\label{eq:lxyfirst}
    l(x,y) := K\left( \|H_\nu(y,x)\| + \|H_\nu(y,x)\|^\beta \right),
\end{equation}
which captures the Hölder-type regularity of \( a \) imposed by Assumption~\ref{assum:1}. Additionally, 
we introduce the function
\begin{equation}\label{eq:lxy}
    l_2(x,y) := \left\{\begin{array}{ll} l(x,y) & \text{under Assumption \ref{assum:2}-1 or \ref{assum:2}-2}\\ 2\|a\|_\infty & \text{under Assumption \ref{assum:2}-3},
    \end{array}\right.
\end{equation}
which satisfies
\[
\left| a\left( \int z \, d\nu^{t,x}(z) \right) - a(y) \right| \leq l_2(x,y).
\]
 The  regularity of \( a \) can now be leveraged to bound the discrepancy term appearing in the covariance difference.
\begin{lemma}\label{lemma:seconduglybound}
Let $\nu^{t,x}$ the probability measure  defined in \eqref{eq:nutx} and $l$,$l_2$ the function defined in \eqref{eq:lxyfirst}, \eqref{eq:lxy}. Then, for all $t>0$, $x\in \mathbb{R}^d$ and $j\in \{1,2\}$, we have
\begin{align*}
    &\frac{Q_tr(x)}{Q_tq(x)}\int  |h^\top H_\nu(y,x)|^{ j}|\frac{e^{a(y)}Q_tq(x)}{Q_tr(x)}-1|d\nu^{t,x}(y)\\
    &\leq C \left(\int  |h^\top H_\nu(y,x)|^{ 2j}d\nu^{t,x}(y)\right)^{1/2}\left(\int\left( e^{l_2(x,y)}l(x,y) e^{a(w)}\right)^{2}d\nu^{t,x}(w)\right)^{1/2}.
\end{align*}
\end{lemma}
The proof of Lemma \ref{lemma:seconduglybound} can be found in Section \ref{sec:lemma:seconduglybound}. We see that the Hölder regularity of the function \( a \) allows us to extract an additional factor of \( \|H_\nu(y,x)\|^\beta \). This gain in exponent translates into stronger concentration, effectively compensating for the insufficient decay in the covariance terms alone. Let us now show that the term $e^{l_2(\cdot,x)}$ remains bounded under the integration of $\nu^{t,x}$.

\begin{lemma}\label{lemma:boundonratios}
For the probability measure $\mu^{t,x}$ being either equal to $p^{t,x}$ \eqref{eq:ptx} or $\nu^{t,x}$ \eqref{eq:nutx} and $l_2$ the function defined in \eqref{eq:lxy}, we have for all $R>0$ that
$$\int e^{Rl_2(x,y)} d\mu^{t,x}(y)\leq C_R.$$
\end{lemma}

The proof of Lemma \ref{lemma:boundonratios} can be found in Section \ref{sec:lemma:boundonratios}. Lemma~\ref{lemma:boundonratios} provides a uniform bound on the exponential moments of the deviation \( H_\mu(y,x) \), where \( \mu^{t,x} \) denotes either the original measure \( p^{t,x} \) or its log-concave approximation \( \nu^{t,x} \). Such control is essential for handling the exponential terms appearing in Lemma~\ref{lemma:seconduglybound}. By combining Lemmas~\ref{lemma:seconduglybound} and~\ref{lemma:boundonratios} with Proposition~\ref{prop:firstuglybound}, we arrive at an upper bound for the second term in equation~\eqref{align:vfdihsksgs2}.

\begin{proposition}\label{prop:lastuglybound} Let  $p^{t,x}$  and $\nu^{t,x}$ be the  probability measures defined in \eqref{eq:ptx} and \eqref{eq:nutx} respectively and $l$ the function defined in \eqref{eq:lxyfirst}. Then, there exists $\Theta\subset (\mathbb{N}_{\geq 0}^3)^3$ such  that for all $t>0$, $x\in \mathbb{R}^d$ and $h\in \mathbb{S}^{d-1}$ we have
\begin{align*}
    |h^\top &\left(\int  H_p(y,x)^{\otimes 2}dp^{t,x}(y)-\int H_\nu(y,x)^{\otimes 2}d\nu^{t,x}(y)\right)h|\\
    & \leq C \sum_{\theta\in \Theta}\prod_{i=1}^3\left(\int  |h^\top H_\nu(y,x)|^{2\theta_{i,1}}d\nu^{t,x}(y)\right)^{\theta_{i,3}/2}\left(\int  l(x,y)^{4\theta_{i,2}}d\nu^{t,x}(y)\right)^{\theta_{i,3}/4},
\end{align*}
with $\sum_{i=1}^3\theta_{i,1}\theta_{i,3}\geq 2 $ and $\bigvee_{i=1}^3\theta_{i,2}\theta_{i,3}= 1$.
\end{proposition}

The proof of Proposition \ref{prop:lastuglybound} can be found in Section \ref{sec:prop:lastuglybound}. Compared to Proposition~\ref{prop:firstuglybound}, Proposition~\ref{prop:lastuglybound} yields a sharper bound by gaining additional control through the factor \( \|H_\nu(y,x)\|^{\beta} \). More precisely, each term in the sum now includes a higher total amounts of moments of \( \|H_\nu(y,x)\| \), reflecting the regularizing effect induced by the Hölder continuity of the perturbation \( a \).
We now extend the Brascamp--Lieb inequality to obtain quantitative bounds on the moments \( \int \|H_\nu(y,x)\|^\gamma \, d\nu^{t,x}(y) \) for any \( \gamma > 0 \).

\begin{lemma}\label{lemma:extendedbl}
Let $\mu=e^{-W}$ a probability measure  such that $\nabla^2 W\succeq \theta \text{Id}$ for some $\theta>0$.  Then, for all $\gamma > 0$ and derivable function $f:\mathbb{R}^d\rightarrow \mathbb{R}$ satisfying $\|\nabla f\|_\infty^2\leq L$,  we have
$$\int |f(y)-\int f(z) d\mu(z)|^\gamma d\mu(y)\leq C_{L,\gamma} \theta^{-\gamma/2}.$$
\end{lemma} 

The proof of Lemma  \ref{lemma:extendedbl} can be found in Section \ref{sec:lemma:extendedbl}. The dependence of the \( \gamma \)-moment on \( \theta^{-\gamma/2} \) in Lemma~\ref{lemma:extendedbl} reflects the Gaussian-type concentration induced by the strong convexity of the potential \( W \), analogous to the behavior of moments under a Gaussian measure with variance \( \theta^{-1} \).

\subsection{Final estimate on  $\lambda_{\max}(\nabla s(t,\cdot))$}
Let us now conclude the proof of Theorem \ref{theo:mainlipreguofscore}. Using Lemma \ref{lemma:extendedbl} we obtain for all $t>0$ and $x\in \mathbb{R}^d$ that
$$\int  |h^\top H_\nu(y,x)|^{2\theta_{i,1}}d\nu^{t,x}(y)\leq C \left(\frac{1-e^{-2t}}{\alpha(1-e^{-2t})+e^{-2t}}\right)^{\theta_{i,1}}$$
and using additionally Jensen's inequality, we get for $\zeta \in \{\beta,1\}$
$$\int \|H_\nu(y,x)\|^{4\zeta \theta_{i,2}}d\nu^{t,x}(y)\leq C \left(\frac{1-e^{-2t}}{\alpha(1-e^{-2t})+e^{-2t}}\right)^{2\zeta\theta_{i,2}}.$$

Now putting this bound in Proposition \ref{prop:lastuglybound}, we obtain for all $h\in \mathbb{S}^{d-1}$ that
\begin{align*}
    |h^\top &\left(\int  H_p(y,x)^{\otimes 2}dp^{t,x}(y)-\int H_\nu(y,x)^{\otimes 2}d\nu^{t,x}(y)\right)h|\nonumber\\
    &\leq C \left(\frac{1-e^{-2t}}{\alpha(1-e^{-2t})+e^{-2t}}\right)^{\frac{1}{2}\sum_{i=1}^3(\theta_{i,1}+\beta\theta_{i,2})\theta_{i,3}}.
\end{align*}
Since the exponents satisfy $\sum_{i=1}^3\theta_{i,1}\theta_{i,2}\geq 2 $ and $\sum_{i=1}^3\beta\theta_{i,2}\theta_{i,3}\geq \beta$, we finally deduce that
\begin{align}\label{eq:finalboundeigenvaluecov}
    |h^\top &\left(\int  H_p(y,x)^{\otimes 2}dp^{t,x}(y)-\int H_\nu(y,x)^{\otimes 2}d\nu^{t,x}(y)\right)h|\leq C \left(\frac{1-e^{-2t}}{\alpha(1-e^{-2t})+e^{-2t}}\right)^{1+\beta/2}.
\end{align}

Therefore, putting this bound together with \eqref{align:vfdihsksgs2}, we conclude
$$\lambda_{\max}\left(\nabla s(t,x)+\frac{\lambda}{\sigma^2}\text{Id} \right)\leq C \frac{e^{-2t}}{(1-e^{-2t})^{1-\beta/2}}.$$

\bibliography{bib}

\appendix
\section{Details of the proofs of Theorem \ref{theo:mainlipreguofscore}}\label{sec:detailpt1}

\subsection{Proofs of the concentration bounds for log-concave measures}

\subsubsection{Proof of Proposition \ref{prop:values}}\label{sec:prop:values}
\begin{proof}
From equation \eqref{align:lawofxt} we have that
\begin{align*}
p_t(x)=&\int (2\pi(1-e^{-2t}))^{-d/2}\exp\left(-\frac{\|x-e^{-t}y\|^2}{2(1-e^{-2t})}\right)p(y)dy\\
    =& \gamma_d(x)Q_tr(x).
\end{align*}
Then, by differentiating under the integral, we obtain
\begin{align*}
    s(t,x)=&\nabla \log p_t(x)\\
    =& -x + \frac{\nabla Q_t r(x)}{Q_t r(x)}\\
    = & -x+\frac{e^{-t}}{1-e^{-2t}}\int (y-e^{-t}x)dp^{t,x}(y)
\end{align*}
and 
\begin{align*}
    \nabla s(t,x)= &-Id+  \frac{\nabla^2 Q_t r(x)}{Q_t r(x)}-\left(\frac{\nabla Q_t r(x)}{Q_t r(x)}\right)^{\otimes 2}\\
    = & -Id+\frac{e^{-2t}}{(1-e^{-2t})^2}\left(\int (y-e^{-t}x)^{\otimes 2}dp^{t,x}(y)-\left(\int (y-e^{-t}x)dp^{t,x}(y)\right)^{\otimes 2}\right)-\frac{e^{-2t}}{1-e^{-2t}}\text{Id}\\
    = &-\text{Id}+\frac{e^{-2t}}{(1-e^{-2t})^2}\int H_p(y,x)^{\otimes 2}dp^{t,x}(y)-\frac{e^{-2t}}{1-e^{-2t}}\text{Id}.
\end{align*}
\end{proof}

\subsubsection{Proof of Lemma \ref{lemma:boundeasy}}\label{sec:lemma:boundeasy}
\begin{proof}
As the probability density $\nu^{0,x}=\frac{q\gamma_d}{\int qd\gamma_d}$  is $\alpha$-log concave, we obtain that 
 $\nu^{t,x}$ is $(\alpha+\frac{e^{-2t}}{1-e^{-2t}})$-log-concave.
Indeed, for all $y\in \mathbb{R}^d$ we have
\begin{equation}\label{eq:ptxislogconc}
-\nabla^2\log(\nu^{t,x}(y))=-\nabla^2\log(q(y)\gamma_d(y))-\nabla^2\log(\varphi^{t,x}(y)/\gamma_d(y))\succeq (\alpha+\frac{e^{-2t}}{1-e^{-2t}}) \text{Id}.
\end{equation}
Then, using the Brascamp-Lieb inequality applied to functions of the form $x\mapsto \langle x,w\rangle $ with $w\in \mathbb{S}^{d-1}$, we obtain
\begin{align*}
    \lambda_{\max}\left(\int  H_\nu(y,x)^{\otimes2} d\nu^{t,x}(y)\right)\leq \lambda_{\max}\left((\alpha+\frac{e^{-2t}}{1-e^{-2t}})^{-1}\text{Id}\right),
\end{align*}
so
\begin{align*}  \frac{e^{-2t}}{(1-e^{-2t})^2}\lambda_{\max}\left(\int  H_\nu(y,x)^{\otimes2} d\nu^{t,x}(y)\right)-\frac{e^{-2t}}{1-e^{-2t}}&\leq \frac{e^{-2t}}{1-e^{-2t}} \frac{1-(\alpha(1-e^{-2t})+e^{-2t})}{\alpha(1-e^{-2t})+e^{-2t}}\\
& = \frac{e^{-2t}(1-\alpha)}{\alpha(1-e^{-2t})+e^{-2t}}.
\end{align*}
\end{proof}

\subsection{Proofs of covariance stability estimates}

\subsubsection{Proof of Proposition \ref{prop:firstuglybound}}\label{sec:prop:firstuglybound}

\begin{proof}

Let us first decompose the difference between the covariances in two different terms. We have
\begin{align}\label{align:jsdujzbsio}
    &\int  H_p(y,x)^{\otimes 2}dp^{t,x}(y)-\int H_\nu(y,x)^{\otimes 2}d\nu^{t,x}(y)\nonumber\\
    &=\int \left( H_p(y,x)^{\otimes 2}\frac{e^{a(y)}}{Q_tr(x)}- H_\nu(y,x)^{\otimes 2}\frac{1}{Q_tq(x)}\right) \exp(-u(\|y\|^2)+\|y\|^2/2)d\varphi^{t,x}(y)\nonumber\\
    &=\int \left( H_p(y,x)^{\otimes 2}- H_\nu(y,x)^{\otimes 2}\right)\frac{e^{a(y)}}{Q_tr(x)}\exp(-u(\|y\|^2)+\|y\|^2/2)d\varphi^{t,x}(y)\nonumber\\
    &+\int  H_\nu(y,x)^{\otimes 2}\left(\frac{e^{a(y)}}{Q_tr(x)}-\frac{1}{Q_tq(x)}\right)\exp(-u(\|y\|^2)+\|y\|^2/2)d\varphi^{t,x}(y)\nonumber\\
    &=\int \left( H_p(y,x)^{\otimes 2}- H_\nu(y,x)^{\otimes 2}\right)\frac{e^{a(y)}Q_tq(x)}{Q_tr(x)}d\nu^{t,x}(y)\nonumber\\
    &+\int  H_\nu(y,x)^{\otimes 2}\left(\frac{e^{a(y)}Q_tq(x)}{Q_tr(x)}-1\right)d\nu^{t,x}(y)\nonumber\\
        &=\int \left( H_p(y,x)^{\otimes 2}- H_\nu(y,x)^{\otimes 2}\right)\left(\frac{e^{a(y)}Q_tq(x)}{Q_tr(x)}-1+1\right)d\nu^{t,x}(y)\nonumber\\
    &+\int  H_\nu(y,x)^{\otimes 2}\left(\frac{e^{a(y)}Q_tq(x)}{Q_tr(x)}-1\right)d\nu^{t,x}(y).
\end{align}
Let us focus on the first term of \eqref{align:jsdujzbsio}. Fixing $x,y\in \mathbb{R}^d$, we have
\begin{align}\label{eq:bhjdshdiss}
    & H_p(y,x)^{\otimes 2}- H_\nu(y,x)^{\otimes 2}\nonumber\\
    = &-H_p(y,x)\otimes \int zdp^{t,x}(z)+\left(\int zd\nu^{t,x}(z)-\int zdp^{t,x}(z)\right)\otimes y+  H_\nu(y,x)\otimes \int zd\nu^{t,x}(z)\nonumber\\
    = &-H_p(y,x)\otimes \left(\int zdp^{t,x}(z)-\int zd\nu^{t,x}(z)\right) +\left(\int zd\nu^{t,x}(z)-\int zdp^{t,x}(z)\right)\otimes  H_\nu(y,x)\nonumber\\
    = &\left(-H_\nu(y,x)+\int zdp^{t,x}(z)-\int zd\nu^{t,x}(z)\right)\otimes \left(\int zdp^{t,x}(z)-\int zd\nu^{t,x}(z)\right)\nonumber\\
    &+\left(\int zd\nu^{t,x}(z)-\int zdp^{t,x}(z)\right)\otimes  H_\nu(y,x)
\end{align}
and
\begin{align}\label{eq:hfsisdosbdoz}
    \int zd\nu^{t,x}(z)-\int zdp^{t,x}(z)& = \int  H_\nu(z,x)d\nu^{t,x}(z)-\int  H_\nu(z,x)dp^{t,x}(z)\nonumber\\
    & = \int  H_\nu(z,x)\left(1-\frac{e^{a(z)}Q_tq(x)}{Q_tr(x)}\right)d\nu^{t,x}(z).
\end{align}

Therefore, taking $h\in \mathbb{S}^{d-1}$ and plugging \eqref{eq:hfsisdosbdoz} into \eqref{eq:bhjdshdiss}, we get
\begin{align*}
&|h^\top\int  (H_p(y,x)^{\otimes 2}- H_\nu(y,x)^{\otimes 2})d\nu^{t,x}(y)h|\\
&\leq \left(\int  h^\top H_\nu(y,x)\left(\frac{e^{a(y)}Q_tq(x)}{Q_tr(x)}-1\right)d\nu^{t,x}(y)\right)^2\nonumber\\
    & + 2|\int  h^\top H_\nu(y,x)d\nu^{t,x}(y)\int  h^\top H_\nu(y,x)\left(\frac{e^{a(y)}Q_tq(x)}{Q_tr(x)}-1\right)d\nu^{t,x}(y)|
\end{align*}
and
\begin{align*}
&|h^\top\int  (H_p(y,x)^{\otimes 2}- H_\nu(y,x)^{\otimes 2})|\frac{e^{a(y)}Q_tq(x)}{Q_tr(x)}-1|d\nu^{t,x}(y)h^\top|\\
&\leq \left(\int  h^\top H_\nu(y,x)\left(\frac{e^{a(y)}Q_tq(x)}{Q_tr(x)}-1\right)d\nu^{t,x}(y)\right)^2\int\left(\frac{e^{a(y)}Q_tq(x)}{Q_tr(x)}-1\right)d\nu^{t,x}(y)\\
    & + 2 \left(\int  h^\top H_\nu(y,x)\left(\frac{e^{a(y)}Q_tq(x)}{Q_tr(x)}-1\right)d\nu^{t,x}(y)\right)^2.
\end{align*}

Finally, plugging these estimates in  \eqref{align:jsdujzbsio}, we obtain
\begin{align}\label{align:brzdijnzide}
    |h^\top&\left(\int  H_p(y,x)^{\otimes 2}dp^{t,x}(y)-\int H_\nu(y,x)^{\otimes 2}d\nu^{t,x}(y)\right)h|\nonumber\\
     \leq &|h^\top\int ( H_p(y,x)^{\otimes 2}- H_\nu(y,x)^{\otimes 2})d\nu^{t,x}(y)h|\nonumber\\
     &+|h^\top\int ( H_p(y,x)^{\otimes 2}- H_\nu(y,x)^{\otimes 2})\left(\frac{e^{a(y)}Q_tq(x)}{Q_tr(x)}-1\right)d\nu^{t,x}(y)h|\nonumber\\
    &+|h^\top\int  H_\nu(y,x)^{\otimes 2}|\frac{e^{a(y)}Q_tq(x)}{Q_tr(x)}-1|d\nu^{t,x}(y)h|\nonumber\\
     \leq &3\left(\int  \|H_\nu(y,x)\|\left(\frac{e^{a(y)}Q_tq(x)}{Q_tr(x)}-1\right)d\nu^{t,x}(y)\right)^2\nonumber\\
    & + 3\int  |h^\top H_\nu(y,x)|d\nu^{t,x}(y)\int  |h^\top H_\nu(y,x)||\frac{e^{a(y)}Q_tq(x)}{Q_tr(x)}-1|d\nu^{t,x}(y)\nonumber\\
    &+ \left(\int  |h^\top H_\nu(y,x)||\frac{e^{a(y)}Q_tq(x)}{Q_tr(x)}-1|d\nu^{t,x}(y)\right)^2\int|\frac{e^{a(y)}Q_tq(x)}{Q_tr(x)}-1|d\nu^{t,x}(y)\nonumber\\
    &+\int  |h^\top H_\nu(y,x)|^{ 2}|\frac{e^{a(y)}Q_tq(x)}{Q_tr(x)}-1|d\nu^{t,x}(y).
\end{align}
We can then conclude that there exists $\Theta\subset (\mathbb{N}_{\geq 0}^3)^3$ such  that
\begin{align*}
    |h^\top &\left(\int  H_p(y,x)^{\otimes 2}dp^{t,x}(y)-\int H_\nu(y,x)^{\otimes 2}d\nu^{t,x}(y)\right)h|\\
    &\leq C \sum_{\theta\in \Theta}\prod_{i=1}^3\left(\int  |h^\top  H_\nu(y,x)|^{\theta_{i,1}}|\frac{e^{a(y)}Q_tq(x)}{Q_tr(x)}-1|^{\theta_{i,2}}d\nu^{t,x}(y)\right)^{\theta_{i,3}},
\end{align*}
with $\sum_{i=1}^3\theta_{i,1}\theta_{i,3}=2$ and $\bigvee_{i=1}^3\theta_{i,2}\theta_{i,3}= 1$.

\end{proof}

\subsubsection{Proof of Lemma \ref{lemma:seconduglybound}}\label{sec:lemma:seconduglybound}

\begin{proof} We have
\begin{align*}
&|\frac{e^{a(y)}Q_tq(x)}{Q_tr(x)}-1|\\
&=\frac{1}{Q_tr(x)}|e^{a(y)}Q_tq(x)-e^{a\left(\int zd\nu^{t,x}(z)\right)}Q_tq(x)+e^{a\left(\int zd\nu^{t,x}(z)\right)}Q_tq(x)-Q_tr(x)|\\
&\leq |e^{a(y)}-e^{a\left(\int zd\nu^{t,x}(z)\right)}|\frac{Q_tq(x)}{Q_tr(x)}+\frac{1}{Q_tr(x)}|\int\left( e^{a\left(\int zd\nu^{t,x}(z)\right)}-e^{a(w)}\right)q(w)d\varphi^{t,x}(w)|.
\end{align*}
Then, as 
\begin{align*}
    |e^{a(y)}-e^{a\left(\int zd\nu^{t,x}(z)\right)}| = &\int_0^1 e^{a(y)+s(a\left(\int zd\nu^{t,x}(z)\right)-a(y))}|a\left(\int zd\nu^{t,x}(z)\right)-a(y)|ds\\
    \leq & C e^{a(y)+l_2(x,y)}l(x,y),
\end{align*}
we deduce that 
\begin{align*}
    &|\frac{e^{a(y)}Q_tq(x)}{Q_tr(x)}-1|\\
    & \leq C \frac{Q_tq(x)}{Q_tr(x)}\left( e^{a(y)+l_2(x,y)}l(x,y) +\int e^{l_2(x,w)}l(x,w) e^{a(w)}d\nu^{t,x}(w)\right).
\end{align*}

Then, for $j\in \{1,2\}$, we have
\begin{align*}
    &\int  |h^\top H_\nu(y,x)|^{ j}|\frac{e^{a(y)}Q_tq(x)}{Q_tr(x)}-1|d\nu^{t,x}(y)\\
    \leq & \int  |h^\top H_\nu(y,x)|^{ j}l(x,w) e^{l_2(x,y)}e^{a(y)}d\nu^{t,x}(y)\frac{Q_tq(x)}{Q_tr(x)}\\
    &+\int  |h^\top H_\nu(y,x)|^{ j}d\nu^{t,x}(y)\int e^{l_2(x,w)}l(x,w) e^{a(w)}d\nu^{t,x}(w)\frac{Q_tq(x)}{Q_tr(x)},
\end{align*}
so applying Cauchy-Schwarz inequality we get the result.

\end{proof}

\subsubsection{Proof of Lemma \ref{lemma:boundonratios}}\label{sec:lemma:boundonratios}

\begin{proof} We show the result for $\mu^{t,x}=p^{t,x}$, the case $\mu^{t,x}=\nu^{t,x}$ follows the same proof.
First, under Assumption \ref{assum:2}-3 the result is immediate as $\|a\|_\infty\leq C$. Likewise, under Assumption \ref{assum:2}-1, the result is immediate as $p^{t,x}$ share the same support than $p$. Now, under Assumption \ref{assum:2}-2, we have
\begin{align*}
    (2\pi(1-e^{-2t})^d)^{1/2}\varphi^{t,x}(y)=&\exp \left(-\frac{1}{2(1-e^{-2t})}\left(\|y\|^2-2\langle y,e^{-t}x\rangle +e^{-2t}\|x\|^2 \right)\right)\\
     =&\exp \left(-\frac{1}{2(1-e^{-2t})}\left(e^{-2t}\|y\|^2-2\langle e^{-t}y,x\rangle +\|x\|^2 + (1-e^{-2t})(\|y\|^2-\|x\|^2) \right)\right)\\
     = & \exp \left(-\frac{\|e^{-t}y-x\|^2}{2(1-e^{-2t})}\right) \gamma_d(y)\gamma_d(x)^{-1},
\end{align*}
so we deduce that
\begin{align*}
    \int e^{Rl(x,y)}dp^{t,x}(y) = \int e^{Rl(x,y)}\frac{e^{-u(y)+a(y)}e^{-\frac{\|e^{-t}y-x\|^2}{2(1-e^{-2t})}}dy}{\int e^{-u(z)+a(z)}e^{-\frac{\|e^{-t}z-x\|^2}{2(1-e^{-2t})}}dz}.
\end{align*}

Define $$f(y)=-u(y)-\frac{\|e^{-t}y-x\|^2}{2(1-e^{-2t})}$$ and 
$$
y^\star\in \argmax_{z\in \mathbb{R}^d} f(z).
$$ 
We have

\begin{align*}
f(y)=&f(y)-f(y^\star)+f(y^\star)\\
&=(y-y^\star)^{T}\int_0^{1}(1-s)\nabla^2 f(y^\star+s(y-y^\star))ds(y-y^\star)+f(y^\star)
\\
&=(y-y^\star)^{T}\int_0^{1}(1-s)(-\nabla^2 u(y^\star+s(y-y^\star))-\frac{e^{-2t}}{1-e^{-2t}}\text{Id})ds(y-y^\star)+f(y^\star)
\end{align*}
so using the assumptions on $\nabla^2 u$, we obtain
\begin{align*}
   -\frac{1}{2}(A+\frac{e^{-2t}}{1-e^{-2t}})\|y-y^\star\|^2+f(y^\star) \leq f(y)\leq -\frac{1}{2}(\alpha+\frac{e^{-2t}}{1-e^{-2t}})\|y-y^\star\|^2+f(y^\star).
\end{align*}
In particular, for all $h\in (0,1)$ we get
\begin{align*}
    \int e^{-u(z)+a(z)}e^{-\frac{\|e^{-t}z-x\|^2}{2(1-e^{-2t})}}dz&\geq\int e^{-u(z)+a(z)}e^{-\frac{\|e^{-t}z-x\|^2}{2(1-e^{-2t})}}\mathds{1}_{\|z-y^\star\|\leq h}dz\\
    &\geq C^{-1} h^de^{f(y^\star)+a(y^\star)-Ch^\beta}e^{-\frac{1}{2}(A+\frac{e^{-2t}}{1-e^{-2t}})h^2}.
\end{align*}
Then, as \begin{align*}
    \int e^{R\|y^\star-y\|}dp^{t,x}(y)=\int e^{R\|y^\star-y\|}\frac{e^{-u(y)+a(y)}e^{-\frac{\|e^{-t}y-x\|^2}{2(1-e^{-2t})}}dy}{\int e^{-u(z)+a(z)}e^{-\frac{\|e^{-t}z-x\|^2}{2(1-e^{-2t})}}dz},
\end{align*}
we deduce that for $k>0$,
\begin{align*}
    \int e^{R\|y^\star-y\|}dp^{t,x}(y)&\leq e^{Rk}+\int e^{R\|y^\star-y\|}\mathds{1}_{\|y-y^\star\|\geq k}dp^{t,x}(y)\\
    &\leq e^{Rk}+C h^{-d}e^{h^\beta} e^{\frac{1}{2}(A+\frac{e^{-2t}}{1-e^{-2t}})h^2}\int e^{R\|y^\star-y\|}e^{-\frac{1}{2}(\alpha+\frac{e^{-2t}}{1-e^{-2t}})\|y-y^\star\|^2}\mathds{1}_{\|y-y^\star\|\geq k}dy\\
    &\leq e^{Rk}+1,
\end{align*}
taking $h=1$ and $k$ large enough depending only on $\alpha,A$ and $\beta$. Now, similarly we obtain
\begin{align}\label{eq:fhskkssyzvaidod}
    \|y^\star-\int y d\nu^{t,x}(y)\|&\leq k+  \int \|y^\star-y\|\mathds{1}_{\|y-y^\star\|\geq k} d\nu^{t,x}(y)\nonumber\\
    &\leq k+1
\end{align}
for $k$ large enough depending only on $\alpha,A$ and $\beta$, so we deduce that
$$l(x,y^\star)\leq C.$$
We can then conclude that
\begin{align*}
    \int e^{Rl(x,y)} d\mu^{t,x}(y) &= \int e^{RK\left( \|H_\nu(y,x)\| + \|H_\nu(y,x)\|^\beta \right)} d\mu^{t,x}(y)\\
    &\leq e^{Rl(x,y^\star)} \int  e^{ RK(\|y-y^\star\|^\beta+\|y-y^\star\|)} d\mu^{t,x}(y)
    \\
    & \leq C_R.
\end{align*}

\end{proof}

\subsubsection{Proof of Proposition \ref{prop:lastuglybound}}\label{sec:prop:lastuglybound}
\begin{proof}
Let us start by bounding the ratio $\frac{Q_tq(x)}{Q_tr(x)}$. We have
\begin{align*}
e^{a(\int z dp^{t,x}(z))}    \frac{Q_tq(x)}{Q_tr(x)}=& \frac{1}{Q_tr(x)}\int e^{a(\int z dp^{t,x}(z))-a(y)}r(y)\varphi^{t,x}(y)dy\\
\leq & \int e^{Cl_2(x,y)}dp^{t,x}(y).
\end{align*}
Then, from Lemma \ref{lemma:boundonratios} we deduce that
\begin{align*}
  \frac{Q_tq(x)}{Q_tr(x)}=& e^{-a(\int z dp^{t,x}(z))} \int e^{Cl_2(x,y)}dp^{t,x}(y)\\
        \leq&  C e^{-a(\int z dp^{t,x}(z))}\\
        \leq&  C e^{-a(\int z d\nu^{t,x}(z))}.
\end{align*}

Therefore,  we obtain that
\begin{align*}
   & \left(\int\left( e^{l_2(x,y)}l(x,y) e^{a(y)}\right)^{2}d\nu^{t,x}(y)\right)^{1/2}\frac{Q_tq(x)}{Q_tr(x)}\\
   &\leq C\left(\int\left( e^{l_2(x,y)}l(x,y) e^{a(y)-a(\int z d\nu^{t,x}(z))}\right)^{2}d\nu^{t,x}(y)\right)^{1/2}\\
      &\leq C \left(\int\left( e^{2l_2(x,y)}l(x,y) \right)^{2}d\nu^{t,x}(y)\right)^{1/2}
      \end{align*}
and applying Cauchy-Schwarz inequality we obtain
\begin{align*}
\int\left( e^{2l_2(x,y)}l(x,y) \right)^{2}d\nu^{t,x}(y)
    & \leq  \left(\int   l(x,y)^4d\nu^{t,x}(y)\right)^{1/2}\left(\int  e^{8 l_2(x,y)}d\nu^{t,x}(y)\right)^{1/2}\\
    & \leq C \left(\int   l(x,y)^4d\nu^{t,x}(y)\right)^{1/2}.
\end{align*}
Then, using Lemma \ref{lemma:seconduglybound} we get
\begin{align*}
    &\int  |h^\top H_\nu(y,x)|^j|\frac{e^{a(y)}Q_tq(x)}{Q_tr(x)}-1|d\nu^{t,x}(y)\\
    &\leq C \left(\int  |h^\top H_\nu(y,x)|^{ 2j}d\nu^{t,x}(y)\right)^{1/2}\left(\int\left( e^{l_2(x,y)}l(x,y) e^{a(w)}\right)^{2}d\nu^{t,x}(w)\right)^{1/2}\frac{Q_tq(x)}{Q_tr(x)}\\
    &\leq C\left(\int  |h^\top H_\nu(y,x)|^{ 2j}d\nu^{t,x}(y)\right)^{1/2}\left(\int l(x,y)^4 d\nu^{t,x}(y)\right)^{1/4}.
\end{align*}
We finally deduce from Proposition \ref{prop:firstuglybound} that
\begin{align*}
    |h^\top &\left(\int  H_p(y,x)^{\otimes 2}dp^{t,x}(y)-\int H_\nu(y,x)^{\otimes 2}d\nu^{t,x}(y)\right)h|\\
    &\leq C \sum_{\theta\in \Theta}\prod_{i=1}^3\left(\int  |h^\top  H_\nu(y,x)|^{\theta_{i,1}}|\frac{e^{a(y)}Q_tq(x)}{Q_tr(x)}-1|^{\theta_{i,2}}d\nu^{t,x}(y)\right)^{\theta_{i,3}}\\
    & \leq C \sum_{\theta\in \Theta}\prod_{i=1}^3\left(\int  |h^\top H_\nu(y,x)|^{2\theta_{i,1}}d\nu^{t,x}(y)\right)^{\theta_{i,3}/2}\left(\int  l(x,y)^{4\theta_{i,2}}d\nu^{t,x}(y)\right)^{\theta_{i,3}/4},
\end{align*}
with $\sum_{i=1}^3\theta_{i,1}\theta_{i,3}\geq 2 $ and  $\bigvee_{i=1}^3\theta_{i,2}\theta_{i,3}= 1$.

\end{proof}

\subsubsection{Proof of Lemma \ref{lemma:extendedbl}}\label{sec:lemma:extendedbl}

\begin{proof}
If $\gamma \in [0,2]$, the result is immediate from Jensen's and Brascamp-Lieb inequalities. Suppose now that $\gamma>2$.
Let $$g(\gamma)=\int |f(y)-\int f(z) d\mu(z)|^\gamma d\mu(y),$$
we have
\begin{align*}
    g(\gamma) =& g(\gamma)-g(\gamma/2)^2+g(\gamma/2)^2\\
    = & \int \left(|f(y)-\int f(z) d\mu(z)|^{\gamma/2}-\int  |f(w)-\int f(z) d\mu(z)|^{\gamma/2}d\mu(w)\right)^2d\mu(y)+g(\gamma/2)^2.
\end{align*}
Now applying the Brascamp-Lieb inequality, we have 
\begin{align*}
    &\int \left(|f(y)-\int f(z) d\mu(z)|^{\gamma/2}-\int  |f(w)-\int f(z) d\mu(z)|^{\gamma/2}d\mu(w)\right)^2d\mu(y)\\
    & \leq \theta^{-1}\left(L \frac{\gamma}{2}\right)^2\int |f(y)-\int f(z) d\mu(z)|^{\gamma-2} d\mu(y),
\end{align*}
so
\begin{align*}
    g(\gamma)\leq L^2 \frac{\gamma^2}{4}\theta^{-1}g(\gamma-2)+g(\gamma/2)^2.
\end{align*}
Using this inequality, let us now show that
$$g(\gamma)\leq  C_{L,\gamma} \theta^{-\gamma/2}.$$
We know that this property is true for $\gamma \in [0,2]$. Suppose now that it is true for all $\gamma \in [0,\eta]$ with $\eta\geq 2$. Then for all $r\in [0,1],$ we have
\begin{align*}
    g(r+\eta)\leq & L^2 \frac{(r+\eta)^2}{4}\theta^{-1}g(r+\eta-2)+g((r+\eta)/2)^2\\
    \leq & L^2 \frac{(r+\eta)^2}{4}\theta^{-1}C_{L,r+\eta-2} \theta^{-(r+\eta-2)/2}+(C_{L,(r+\eta)/2} \theta^{-((r+\eta)/2)/2})^2\\
    \leq & C_{L,r+\eta}\theta^{-(r+\eta)/2},
\end{align*}
so by induction the property is true for all $\gamma\geq 0$.
\end{proof}

\subsection{Additional proofs of Lipschitz estimates}

\subsubsection{Proof of Proposition \ref{prop:SDEstability-drift}}\label{sec:prop:SDEstability-drift}
\begin{proof}
Let us take $X_0\in \mathbb{R}^d$ and $X_t$,$\bar{X}_t$ defined by
\begin{align*}
    X_t & = X_0 + \int_0^t a(s,X_s) ds + \int_0^t b_s dB_s \\
    \bar{X}_t & = X_0 + \int_0^t \bar{a}(s,\bar{X}_s) ds + \int_0^t b_s dB_s,
\end{align*}
for the same $B_t$ so that $X_t$ and $\bar{X}_t$ are solutions to the SDE's. In particular this defines a coupling for $(p_t,\bar{p}_t)$ so we deduce that $$W_2(p_t,\bar{p}_t)^2 \leq \mathbb{E}(\|X_t - \bar{X}_t\|^2).$$ Denoting $g(s)=\mathbb{E}(\|X_s - \bar{X}_s\|^2)$ and $L_s=\sup_{x \in \mathbb{R}^d} \lambda_{\max}(\nabla a_s(x))$, it holds that
\begin{align*}
    g'(s) & = \mathbb{E} \left ( 2 \left\langle \bar{a}(s,\bar{X}_s) - a(s,X_s),\bar{X}_s - X_s \right\rangle \right ) \\
    & = 2 \mathbb{E} \left ( \left\langle \bar{a}(s,\bar{X}_s) - a(s,\bar{X}_s),\bar{X}_s - X_s \right\rangle \right ) + 2 \mathbb{E} \left ( \left\langle a(s,\bar{X}_s) - a(s,X_s),\bar{X}_s - X_s \right\rangle \right ) \\
    & \leq 2 \mathbb{E} \left ( \left\langle \bar{a}(s,\bar{X}_s) - a(s,\bar{X}_s),\bar{X}_s - X_s \right\rangle \right ) + 2L_s g(s) \\
    & \leq  2\mathbb{E} \left ( \left \|\bar{a}(s,\bar{X}_s) - a(s,\bar{X}_s) \right \|^2 \right )^{1/2}g(s)^{1/2} + 2 L_sg(s).
\end{align*}
We deduce from the non linear Gronwall's inequality (Theorem 21 in \cite{dragomir2003some}) that
\begin{align*}
    g(t)^{1/2} \leq &  \int_0^t \exp \left ( \int_\tau^t L_s ds \right )\mathbb{E} \left ( \left \|\bar{a}(\tau,\bar{X}_\tau) - a(\tau,\bar{X}_\tau) \right \|^2 \right )^{1/2}d\tau\\
    \leq &  \left(t\int_0^t \exp \left ( 2\int_\tau^t L_s ds \right )\mathbb{E} \left ( \left \|\bar{a}(\tau,\bar{X}_\tau) - a(\tau,\bar{X}_\tau) \right \|^2 \right )d\tau\right)^{1/2}.
\end{align*}
Therefore, we obtain that
\begin{align*}
    W_2^2(p_t,\bar{p}_t)\leq t\int_0^t \exp \left ( 2\int_\tau^t L_s ds \right )\int \left \|\bar{a}(\tau,x) - a(\tau,x) \right \|^2 d\bar{p}_\tau(x)d\tau.
\end{align*}
\end{proof}

\subsubsection{Proof of Proposition \ref{prop:unboundjaccompact}}\label{sec:prop:unboundjaccompact}
\begin{proof}
From Proposition~\ref{prop:values}, we have the identity
\begin{align*}
s(t,x) - s(t,y) = y - x + \frac{e^{-t}}{1 - e^{-2t}} \left( e^{-t}(y - x) + \int z \left( p^{t,x}(z) - p^{t,y}(z) \right) \, dz \right),
\end{align*}
which yields
\begin{align*}
\langle s(t,x) - s(t,y), y - x \rangle 
&= \left( 1 + \frac{e^{-2t}}{1 - e^{-2t}} \right) \|x - y\|^2 +\frac{e^{-t}}{1 - e^{-2t}} \left\langle \int z \left( p^{t,x}(z) - p^{t,y}(z) \right) \, dz, y - x \right\rangle \\
&\geq \left( 1 + \frac{e^{-2t}}{1 - e^{-2t}} \right) \|x - y\|^2 - C\frac{e^{-t}}{1 - e^{-2t}} \|x - y\|,
\end{align*}
since \( p^{t,x} \) and \( p^{t,y} \) have the same support than \( p \). Taking \( x \) and \( y \) sufficiently far apart, we deduce that
\[
\|\nabla s(t, \cdot)\|_\infty \geq 1 + \frac{e^{-2t}}{1 - e^{-2t}} \geq 1 + C e^{-2t}(1 + t^{-1}).
\]
\end{proof}

\subsubsection{Proof of Corollary \ref{coro:lipschitzbound}}\label{sec:coro:lipschitzbound}

\begin{proof}
The bound on $\lambda_{\max}(\nabla s(t,x))$ is given by Theorem \ref{theo:mainlipreguofscore}. The bound on $\lambda_{\min}(\nabla s(t,x))$ follows the same ideas but uses the Cramer-Rao bound. Specifically, similarly to \eqref{align:vfdihsksgs2} we have
\begin{align}\label{align:vfdihsksgs2min}
\lambda_{\min}\left( \nabla s(t,x) + \mathrm{Id} \right)
&= \lambda_{\min}\left( \frac{e^{-2t}}{(1 - e^{-2t})^2} \int H_p(y,x)^{\otimes 2} \, dp^{t,x}(y) - \frac{e^{-2t}}{1 - e^{-2t}} \mathrm{Id} \right) \nonumber \\
&\geq \frac{e^{-2t}}{(1 - e^{-2t})^2} \lambda_{\min}\left( \int H_p(y,x)^{\otimes 2} \, dp^{t,x}(y) - \int H_\nu(y,x)^{\otimes 2} \, d\nu^{t,x}(y) \right) \nonumber \\
&\quad + \frac{e^{-2t}}{(1 - e^{-2t})^2} \lambda_{\min}\left( \int H_\nu(y,x)^{\otimes 2} \, d\nu^{t,x}(y) \right) - \frac{e^{-2t}}{1 - e^{-2t}}.
\end{align}
Now, from \eqref{eq:finalboundeigenvaluecov} we have
\begin{align*}
    &\frac{e^{-2t}}{(1 - e^{-2t})^2} \lambda_{\min}\left( \int H_p(y,x)^{\otimes 2} \, dp^{t,x}(y) - \int H_\nu(y,x)^{\otimes 2} \, d\nu^{t,x}(y) \right)\\
    &\geq -C\frac{e^{-2t}}{(1 - e^{-2t})^2}\left(\frac{1-e^{-2t}}{\alpha(1-e^{-2t})+e^{-2t}}\right)^{1+\beta/2}\\
    &\geq -C\frac{e^{-2t}}{(1 - e^{-2t})^{1-\beta/2}}.
\end{align*}
On the other hand, using Cramer-Rao bound we have
$$\lambda_{\min}\left(  \int H_\nu(y,x)^{\otimes 2} \, d\nu^{t,x}(y) \right)\geq \frac{1-e^{-2t}}{A(1-e^{-2t})+e^{-2t}},$$
so we deduce like for Lemma \ref{lemma:boundeasy} that
\begin{align*}
    \frac{e^{-2t}}{(1 - e^{-2t})^2} \lambda_{\min}\left( \int H_\nu(y,x)^{\otimes 2} \, d\nu^{t,x}(y) \right) - \frac{e^{-2t}}{1 - e^{-2t}}\geq \frac{e^{-2t}(1-A)}{A(1-e^{-2t})+e^{-2t}}.
\end{align*}
Therefore, we finally have
\begin{align*}
    \lambda_{\min}\left( \nabla s(t,x) + \mathrm{Id} \right)&\geq \frac{e^{-2t}(1-A)}{A(1-e^{-2t})+e^{-2t}}-C\frac{e^{-2t}}{(1 - e^{-2t})^{1-\beta/2}}\\
    & \geq - Ce^{-2t}(1+t^{\frac{\beta}{2}-1}).
\end{align*}
\end{proof}
 
\subsubsection{Proof of Proposition \ref{prop:liptran}}\label{sec:prop:liptran}
 
\begin{proof}
For $t\in (0,1]$, let us define
\begin{equation}\label{eq:V}
    V(t,x):=\frac{1}{t} \nabla \log Q_{\log(t^{-1})} r(x)=\frac{1}{t}\big(x+s(\log(t^{-1}),x)\big)
\end{equation}
with
\begin{equation}\label{eq:Qtr2}
    Q_{\log(t^{-1})} r(x)=  \int r(tx+\sqrt{1-t^2}z)d\gamma_d(z).
\end{equation}
From Proposition 2 in  \cite{stephanovitch2024smooth} we have that the transport map of the Probability flow ODE is the solution at time $t=1$ to the equation \begin{equation}\label{eq:PDE}
\partial_t X_t(x) =V(t,X_t(x)), \quad X_0(x)=x.
\end{equation}
Then, using Corollary \ref{coro:finiteintegrabilityofscore} we have that
\begin{equation*}
\int_0^1 \sup_{x\in \mathbb{R}^d} \lambda_{\max}\big(\nabla V_t(x)\big)\leq C,
\end{equation*}
so from Lemma 2 in \cite{stephanovitch2024smooth} we get the result.
\end{proof}
 
\section{Proof of Theorem \ref{theo:highregu}}
\label{sec:theo:highregu}
This section is dedicated to the proof of Theorem \ref{theo:highregu}.
In order to compute the higher differentials of the score, we make use of the Faa di Bruno formula. For two unit-variate functions $f,g:\mathbb{R}\rightarrow \mathbb{R}$, the Faa di Bruno formula states that the $k$-derivative of the composition of $f$ and $g$ is
\begin{equation*}
    (f\circ g)^{(k)}(x) =\sum_{\pi \in \Pi_k} f^{(|\pi|)} (g(x)) \prod_{B\in \pi} g^{(|B|)} (x),
\end{equation*}
where $\Pi_k$ denotes the set of all partitions of the set $\{1,...,k\}$. For simplicity, in the multivariate setting we will keep this notation by writing for $f,g:\mathbb{R}^d\rightarrow \mathbb{R}^d$,
$$\nabla^k (f\circ g)(x)=\sum_{\pi \in \Pi_k} \nabla^{|\pi|}f(g(x)) \prod_{B\in \pi} \nabla^{|B|} g(x),$$
where the latter is defined by
\begin{align*}
&\left(\sum_{\pi \in \Pi_k} \nabla^{|\pi|}f(g(x)) \prod_{B\in \pi} \nabla^{|B|} g(x)\right)w^k\\
& := \sum_{\pi=(B_1,B_2,...,B_i) \in \Pi_k} \nabla^{|\pi|}f(g(x)) \left( \nabla^{|B_1|} g(x)w^{|B_1|},\nabla^{|B_2|}g(x)w^{|B_2|},...,\nabla^{|B_i|} g(x)w^{|B_i|}\right).
\end{align*}

\subsection{Regularity for large time $t$}\label{sec:regulargetime}

We begin by considering the case where \( t \) is not too small; that is, we fix a constant \( C^\star > 0 \) and assume \( t \geq \log(\epsilon^{-1})^{-C^\star} \). Throughout the proof, the value of \( C^\star \) may be increased as needed, but it will always remain a constant independent of \( \epsilon \).

Let us define
\begin{equation}\label{eq:ystar}
    y^\star = \argmin_y u(y)
\end{equation}
 and take the set 
 \begin{equation}
     A_t^\epsilon = A_\infty^\epsilon,
 \end{equation}
 with
\begin{equation}\label{eq:aepstsmall}
A_\infty^\epsilon:=\{y\in \mathbb{R}^d\big| \ \|y-y^\star\|^2\leq C^\star \big(\log(\epsilon^{-1})(1+\alpha^{-1})+K\big)\}.
\end{equation}
Let us first show that this set contains at least $1-\epsilon$ of the mass.
\begin{proposition}\label{prop:masstlarge}
 There exists $C^\star>0$ large enough such that the set $A_\infty^\epsilon$ defined in \eqref{eq:aepstsmall} satisfies for all $t\geq 0$ that
 $p_t(A_\infty^\epsilon)\geq 1-\epsilon.$
\end{proposition}
The proof of Proposition \ref{prop:masstlarge} can be found in Section \ref{sec:prop:masstlarge}.
Let us now compute the $k$-differential of $s(t,\cdot)$ for $k\in \mathbb{N}_{\geq 0}$.
From the Faa Di Bruno formula we have 
\begin{align}\label{align:gradvtpetit}
    \nabla^{k}\left( s(t,\cdot)+\text{Id}\right)(x)&=\sum_{\pi \in \Pi_{k+1}} \frac{(-1)^{|\pi|+1}}{(Q_{t} r(x))^{|\pi|}} \prod_{B\in \pi} \nabla^{|B|} Q_{t} r(x)\nonumber \\
    & =\sum_{\pi \in \Pi_{k+1}} (-1)^{|\pi|+1} \prod_{B\in \pi} \frac{1}{Q_{t} r(x)} \int \nabla^{|B|}_x\varphi^{t,x}(y)r(y)dy\nonumber\\
    & = \frac{e^{-(k+1)t}}{(1-e^{-2t})^{k+1}}\sum_{\pi \in \Pi_{k+1}} (-1)^{|\pi|+1} \prod_{B\in \pi} \int F_{|B|}(y-e^{-t}x)dp^{t,x}(y),
\end{align}
with $$ F_{|B|}(z):=(1-e^{-2t})^{|B|}\exp(\frac{\|z\|^2}{1-e^{-2t}})\nabla^{|B|} \exp(-\frac{\|z\|^2}{1-e^{-2t}}).$$

Let us give a bound on the quantity $\int F_{|B|}(y-e^{-t}x)dp^{t,x}(y)$.

\begin{lemma}\label{lemma:concentrationsmallthigh}
There exists $C^\star>0$ such that for all $l\in \mathbb{N}_{>0}$  and $x\in A_t^\infty$ defined in \eqref{eq:aepstsmall}, we have
$$\|\int F_{l}(y-e^{-t}x)dp^{t,x}(y)\|\leq C\log(\epsilon^{-1})^{l C_2}.$$
\end{lemma}
The proof of Lemma \ref{lemma:concentrationsmallthigh} can be found in Section \ref{sec:lemma:concentrationsmallthigh}. Combining this result with \eqref{align:gradvtpetit},  we can directly conclude the case of large time $t$.

\begin{proposition}\label{prop:fjidididjhdhzpmwjf}
For all $k\in \mathbb{N}$, $\epsilon\in (0,1)$, $t>0$ and $x\in A_\infty^\epsilon$ defined in \eqref{eq:aepstsmall}, we have
\begin{equation*}
    \|\nabla^{k}\left( s(t,\cdot)+\text{Id}\right)(x)\| \leq C\log(\epsilon^{-1})^{C_2(1+k)}  \frac{e^{-(k+1)t}}{(1-e^{-2t})^{k+1}}.
\end{equation*}
In particular, if $t\geq \log(\epsilon^{-1})^{-C^\star}$ we have
\begin{equation*}
    \|\nabla^{k}\left( s(t,\cdot)+\text{Id}\right)(x)\| \leq C\log(\epsilon^{-1})^{(C^\star+C_2)(1+k)} e^{-(k+1)t}.
\end{equation*}
\end{proposition}

\subsection{Regularity for small time $t$}
By opposition to Section \ref{sec:regulargetime}, let us now consider the case $0<t\leq \log(\epsilon^{-1})^{-C^\star}$. We take
\begin{equation}
    A^\epsilon_t=A^\epsilon_0
\end{equation}
with
\begin{align}\label{align:aepst}
A^\epsilon_0:=\{y\in \mathbb{R}^d|\ u(y)\leq K \big(\log(\epsilon^{-1})+R+1\big)^K\},
\end{align}
for a certain $R>K$. Let us fist show that this set contains at least $1-\epsilon$ of the mass.
\begin{proposition}\label{prop:masstsmall}
 There exists $C^\star,R>0$ large enough such that the set $A_0^\epsilon$ defined in \eqref{align:aepst} satisfies for all $0\leq t\leq \log(\epsilon^{-1})^{-C^\star}$ that
 $p_t(A_0^\epsilon)\geq 1-\epsilon.$
\end{proposition}
The proof of Proposition \ref{prop:masstsmall} can be found in Section \ref{sec:prop:masstsmall}.
In the following, we treat separately the bound on $\|\nabla^\gamma s(t,\cdot)\|$ in the cases $\gamma=0$, $\gamma=1$ and $\gamma\geq 2$.

\subsubsection{Bound on $s(t,\cdot)$ for $\beta\leq 1$}\label{sec:boundonst}

Let us start with the bound on the infinite norm of $s(t,\cdot)$, we have from Proposition \ref{prop:values} that
\begin{align}\label{eq:vghjxhzz}
s(t,x)=-x+\frac{e^{-t}}{1-e^{-2t}}\int (y-e^{-t}x)dp^{t,x}(y).
\end{align}

Now, using that
$$\int (y-e^{-t}x)\varphi^{t,x}(y)dy=0$$
and applying Cauchy-Schwarz inequality, we obtain
\begin{align}\label{eq:premconcentpetit}
    \|\int (y-e^{-t}x) dp^{t,x}(y)& = \|\int (y-e^{-t}x) \frac{r(y)-r(e^{-t}x)}{Q_tr(x)}\varphi^{t,x}(y)dy\|\nonumber\\
    &\leq \left(\int \|y-e^{-t}x\|^2\varphi^{t,x}(y)dy\right)^{1/2}\left( \int \left( \frac{r(y)-r(e^{-t}x)}{Q_tr(x)}\right)^{2}\varphi^{t,x}(y)dy\right)^{1/2}\nonumber\\
    &\leq (1-e^{-2t})^{1/2}\left( \int \left( \frac{r(y)-r(e^{-t}x)}{Q_tr(x)}\right)^{2}\varphi^{t,x}(y)dy\right)^{1/2}.
\end{align}
Let us now extract an additional $(1-e^{-2t})^{\beta \wedge 1}$ factors from \eqref{eq:premconcentpetit} using the following Lemma.
\begin{lemma}\label{lemma:boundsquared} There exists $C^\star>0$ such that for all $t\in (0, \log(\epsilon^{-1})^{-C^\star})$ and $A_0^\epsilon$ defined in \eqref{align:aepst}, we have
\begin{align*}
    \int \left( \frac{r(y)-r(e^{-t}x)}{Q_tr(x)}\right)^{2}\varphi^{t,x}(y)dy\leq C\log(\epsilon^{-1})^{C_2}(1-e^{-2t})^{\beta\wedge 1}.
\end{align*}
\end{lemma}
The proof of Lemma \ref{lemma:boundsquared} can be found in Section \ref{sec:lemma:boundsquared}.
Using this result, the value of the score \eqref{eq:vghjxhzz} and the estimate \eqref{eq:premconcentpetit}, we can conclude the bound on $s(t,\cdot)$ in the case where the time $t$ is small and $\beta\leq 1$.
\begin{proposition}
 There exists $C^\star>0$ such that for all $t\in (0, \log(\epsilon^{-1})^{-C^\star})$ and $x\in A_0^\epsilon$ defined in \eqref{align:aepst}, we have
 $$\|s(t,x)+x\|\leq C\log(\epsilon^{-1})^{C_2}e^{-t} (1-e^{-2t})^{-\frac{1}{2}(1-(\beta\wedge 1))}.$$
\end{proposition}

\subsubsection{Bound on $\nabla s(t,\cdot)$ for $\beta\leq 1$}

We now provide a bound on $\|\nabla s(t,\cdot)\|$ using a similar method than for $\| s(t,\cdot)\|$. From Proposition \ref{prop:values} we have
\begin{align}\label{eq:todifferentiatewhenbeta0}
    &\nabla s(t,x)+\text{Id}\nonumber\\
    =&\frac{e^{-2t}}{(1-e^{-2t})^2}\int H_p(y,x)^{\otimes 2}dp^{t,x}(y)-\frac{e^{-2t}}{1-e^{-2t}}\text{Id}\nonumber\\
    =&\frac{e^{-2t}}{(1-e^{-2t})^2}\int \left((y-e^{-t}x)^{\otimes 2}+(y-e^{-t}x)\otimes \left(e^{-t}x-\int z dp^{t,x}(z)\right)\right)dp^{t,x}(y)-\frac{e^{-2t}}{1-e^{-2t}}\text{Id}\nonumber\\
    =&\frac{e^{-2t}}{(1-e^{-2t})^2}\left(\int \left((y-e^{-t}x)^{\otimes 2}-(1-e^{-2t})\text{Id}\right)dp^{t,x}(y)-\left(\int (z-e^{-t}x) dp^{t,x}(z)\right)^{\otimes 2}\right).
\end{align}
The term
$\|\left(\int (z-e^{-t}x) dp^{t,x}(z)\right)^{\otimes 2}\|$ can be bounded 
using \eqref{eq:premconcentpetit} and Lemma \ref{lemma:boundsquared}. On the other hand, using Cauchy-Schwarz and Brascamp-Lieb inequalities we obtain
\begin{align*}
    \|&\frac{e^{-2t}}{(1-e^{-2t})^2}\int \left((y-e^{-t}x)^{\otimes 2}-(1-e^{-2t})\text{Id}\right)dp^{t,x}(y)\|\\
    =&\frac{e^{-2t}}{(1-e^{-2t})^2}\|\int\left( (y-e^{-t}x)^{\otimes 2}-(1-e^{-2t})\text{Id}\right)\frac{r(y)-r(e^{-t}x)}{Q_tr(x)}\varphi^{t,x}(y)dy\|\\
    \leq & \frac{Ce^{-2t}}{(1-e^{-2t})^2}\left(\int\| (y-e^{-t}x)^{\otimes 2}-(1-e^{-2t})\text{Id}\|^2\varphi^{t,x}(y)dy\right)^{1/2}\left(\int\left(\frac{r(y)-r(e^{-t}x)}{Q_tr(x)}\right)^2\varphi^{t,x}(y)dy\right)^{1/2}
    \\
    \leq & \frac{Ce^{-2t}}{(1-e^{-2t})^{3/2}}\left(\int\| y-e^{-t}x\|^2\varphi^{t,x}(y)dy\right)^{1/2}\left(\int\left(\frac{r(y)-r(e^{-t}x)}{Q_tr(x)}\right)^2\varphi^{t,x}(y)dy\right)^{1/2}
    \\
    \leq & \frac{Ce^{-2t}}{1-e^{-2t}}\left(\int\left(\frac{r(y)-r(e^{-t}x)}{Q_tr(x)}\right)^2\varphi^{t,x}(y)dy\right)^{1/2}.
\end{align*}
Finally, using again Lemma \ref{lemma:boundsquared} we can directly conclude the bound in the case $\beta\leq 1$.
\begin{proposition}
 There exists $C^\star>0$ such that for all $t\in (0, \log(\epsilon^{-1})^{-C^\star})$ and $x\in A_0^\epsilon$ defined in \eqref{align:aepst}, we have
$$\|\nabla s(t,x)+\text{Id}\|\leq C\log(\epsilon^{-1})^{C_2}e^{-2t}(1-e^{-2t})^{-(1-\frac{\beta\wedge 1}{2})}.$$
\end{proposition}

\subsubsection{Bound on $\nabla^j s(t,\cdot)$ for $j\in \mathbb{N}_{\geq 0}$ and $\beta\geq 1$}
Let us now give a bound on the norm of the higher differentials of $s(t,\cdot)$.
Taking $l\in \mathbb{N}_{>0}$, from the Faa di Bruno formula we have 
\begin{align}\label{eq:nablajlogQ}
    \nabla^l\log(Q_tr(x))&=\sum_{\pi \in \Pi_{l+1}}(-1)^{|\pi|+1}  \prod_{B\in \pi} \frac{\nabla^{|B|} Q_{t} r(x)}{Q_{t} r(x)}
\end{align}
and writing $\gamma_d$ for the standard $d$-dimensional Gaussian, we have  for $l\leq \lfloor \beta \rfloor$ that
\begin{align}\label{eq:deruptobeta}
     \frac{\nabla^l Q_{t} r(x)}{Q_{t} r(x)}=\frac{1}{Q_{t} r(x)}\int \nabla_x^l r(e^{-t}x+\sqrt{1-e^{-2t}}y)d\gamma_d(y)=e^{-tl}\int \frac{\nabla^lr(y)}{r(y)} dp^{t,x}(y)dy.
\end{align}

Let us start by giving a bound on the differentials of degree lower than $\lfloor \beta \rfloor$.

\begin{proposition}\label{prop:enfaiteccestcoro} There exists $C^\star>0$ such that if $t\leq \log(\epsilon^{-1})^{-C^\star}$ and $x\in A_t^\epsilon$ defined in \eqref{align:aepst}, then for all $i\in \{1,...,\lfloor \beta \rfloor\}$ we have
 $$\|\int \frac{\nabla^ir(y)}{r(y)} dp^{t,x}(y)dy\|\leq C\log(\epsilon^{-1})^{C_2}.$$
\end{proposition}
The proof of Proposition \ref{prop:enfaiteccestcoro} can be found in Section \ref{sec:prop:enfaiteccestcoro}. In particular, using \eqref{eq:nablajlogQ} and \eqref{eq:deruptobeta}, Proposition \ref{prop:enfaiteccestcoro} allows us to conclude the bound on $\|\nabla^{j} s(t,\cdot)\|$ for $j\in \{0,...,\lfloor \beta \rfloor-1\}$ in the case $\beta\geq 1$.
\begin{proposition}
 There exists $C^\star>0$ such that for all $t\in (0, \log(\epsilon^{-1})^{-C^\star})$, $x \in A_0^\epsilon$ defined in \eqref{align:aepst} and $j\in \{0,...,\lfloor \beta \rfloor-1\}$, we have
\begin{equation}
   \|\nabla^{j}\big( s(t,\cdot)+\text{Id}\big)(x)\|\leq  C\log(\epsilon^{-1})^{C_2}e^{-t(j+1)}.
\end{equation}
\end{proposition}

Now, in order to compute the differentials of \eqref{eq:deruptobeta}, let us define for $m\in \mathbb{N}_{>0}$ and $y\in \mathbb{R}^d$,
\begin{equation}\label{eq:Yi}
Y_1:=\frac{y-\int zdp^{t,x}(z)}{1-e^{-2t}}\quad  \text{and} \quad Y_{m+1}:=\frac{\int \left(z-\int w dp^{t,x}(w)\right)^{\otimes (m+1)}dp^{t,x}(z)}{(1-e^{-2t})^{m+1}}.
\end{equation}

Using these quantities, let us derive the differentials of \eqref{eq:deruptobeta}.
\begin{proposition}\label{prop:thepoly}
For all $j\in \mathbb{N}_{>0}$ and $l\in \{1,...,\lfloor \beta \rfloor\}$, we have 
\begin{align}\label{align:nablajqsurq}
    \nabla^{j}\left(z\mapsto \frac{\nabla^{l} Q_{t} r(z)}{Q_{t} r(z)}\right)(x)=e^{-t(j+l)}\int \frac{\nabla^{l}r(y)}{r(y)}\otimes\left( P^{j}(y,x)-\int P^{j}(z,x)dp^{t,x}(z) \right)dp^{t,x}(y),
\end{align}
with $P^{j}(y)$ a polynomial of the form
\begin{align*}
    P^j(y,x)=\sum_{l=1}^m\sum_{\substack{i_1,...,i_l\in \{1,...,j\}\\i_1+...+i_l=j}}a_{i_1,...,i_l} \bigotimes_{n=1}^lY_{i_n}
\end{align*}
and $a_{i_1,...,i_j}\in \mathbb{R}$.
\end{proposition}
The proof of Proposition \ref{prop:thepoly} can be found in Section \ref{sec:prop:thepoly}. We have from \eqref{eq:nablajlogQ} that for $j\geq \lfloor \beta \rfloor\vee 1$,
\begin{align}\label{eq:nablajscoregene}
    \nabla^js(t,x)&=\nabla^{j+1-\lfloor \beta \rfloor}\left(z\mapsto \sum_{\pi \in \Pi_{\lfloor \beta \rfloor}}(-1)^{|\pi|+1}  \prod_{B\in \pi} \frac{\nabla^{|B|} Q_{t} r(z)}{Q_{t} r(z)}\right)(x)\nonumber\\
    & = \sum_{\pi \in \Pi_{\lfloor \beta \rfloor}}(-1)^{|\pi|+1}  \sum_{\substack{i_{\pi_1},...,i_{\pi_{|B|}}\\i_{\pi_1}+...+i_{\pi_{|B|}}=j+1-\lfloor \beta \rfloor}} \prod_{B\in \pi}\nabla^{i_B}\left(z\mapsto \frac{\nabla^{|B|} Q_{t} r(z)}{Q_{t} r(z)}\right)(x).
\end{align}
Then, using Proposition \ref{prop:thepoly} we get
\begin{align*}
    &\|\nabla^{i_B}\left(z\mapsto \frac{\nabla^{|B|} Q_{t} r(z)}{Q_{t} r(z)}\right)(x)\|\\
    &\leq e^{-t(i_B+|B|)}\|\int \left(\frac{\nabla^{|B|}r(y)}{r(y)}-\int \frac{\nabla^{|B|}r(z)}{r(z)}dp^{t,x}(z)\right)\otimes\left( P^{i_B}(y,x)-\int P^{i_B}(z,x)dp^{t,x}(z) \right)dp^{t,x}(y)\|
\end{align*}
which gives using Cauchy Schwarz inequality
\begin{align}\label{eq:Chmoche}
    \|&\nabla^{i_B}\left(z\mapsto \frac{\nabla^{|B|} Q_{t} r(z)}{Q_{t} r(z)}\right)(x)\|\nonumber\\
    \leq & e^{-t(i_B+|B|)}\left(\int \|\frac{\nabla^{|B|}r(y)}{r(y)}-\int \frac{\nabla^{|B|}r(z)}{r(z)}dp^{t,x}(z)\|^2dp^{t,x}(y)\right)^{1/2}\nonumber\\
    & \times \left( \int\|P^{i_B}(y,x)-\int P^{i_B}(z,x)dp^{t,x}(z) \|^2dp^{t,x}(y)\right)^{1/2}.
\end{align}

In order to bound the second term of \eqref{eq:Chmoche}, let us first quantify the concentration of the mass of $p^{t,x}$ around the point $e^{-t}x$.

\begin{lemma}\label{lemma:concetrationaroudx}
There exists $C^\star>0$ such that, for all $\eta>0$, $t\in (0, \log(\epsilon^{-1})^{C^\star})$ and $A_0^\epsilon$ defined in \eqref{align:aepst}, we have
\begin{align*}
    \int&\|y-e^{-t}x\|^{\eta}dp^{t,x}(y)\leq C_\eta\log(\epsilon^{-1})^{C_2}(1-e^{-2t})^{\eta/2}
\end{align*}
\end{lemma}
The proof of Lemma \ref{lemma:concetrationaroudx} can be found in Section \ref{sec:lemma:concetrationaroudx}. Using Lemma \ref{lemma:concetrationaroudx}, we have for $m\geq 2$ that
\begin{align*}
    \|Y_m\| &=\|\frac{\int \left(z-\int w dp^{t,x}(w)\right)^{\otimes m}dp^{t,x}(z)}{(1-e^{-2t})^m}\|\leq C_\eta\log(\epsilon^{-1})^{C_2}(1-e^{-2t})^{-m/2},
\end{align*}
so as for $L\in \mathbb{N}_{>0}$ we have
\begin{align*}
    &\int\|P^{L}(y,x)-\int P^{L}(z,x)dp^{t,x}(z) \|^2dp^{t,x}(y)\\
    &\leq C \sum_{l=1}^m\sum_{\substack{i_1,...,i_l\in \{1,...,L\}\\i_1+...+i_l=L}}|a_{i_1,...,i_l}|\int \| \bigotimes_{n=1}^lY_{i_n}(y)-\int \bigotimes_{n=1}^lY_{i_n}(z)dp^{t,x}(z)\|^2dp^{t,x}(y),
\end{align*}
we deduce applying the Brascamp-Lieb inequality and Lemma \ref{lemma:extendedbl} that
\begin{align*}
    &\int \| \bigotimes_{n=1}^lY_{i_n}(y)-\int \bigotimes_{n=1}^lY_{i_n}(z)dp^{t,x}(z)\|^2dp^{t,x}(y)\\
    &\leq \|\bigotimes_{\substack{n=1\\i_n\geq 2}}^lY_{i_n}(y)\|^2\int \| \bigotimes_{\substack{n=1\\i_n=1}}^lY_{i_n}(y)-\int \bigotimes_{\substack{n=1\\i_n=1}}^lY_{i_n}(z)dp^{t,x}(z)\|^2dp^{t,x}(y)\\
    &\leq C_L\log(\epsilon^{-1})^{C_2} (1-e^{-2t})^{-\sum_{k=1}^l i_l\mathds{1}_{i_l\geq 2}}\\
    & \times \int \| \left(Y_{1}(y)\right)^{-\otimes\sum_{k=1}^l i_l\mathds{1}_{i_l=1}}-\int \left(Y_{1}(z)\right)^{-\otimes\sum_{k=1}^l i_l\mathds{1}_{i_l=1}}(z)dp^{t,x}(z)\|^2dp^{t,x}(y)\\
    &\leq C_L\log(\epsilon^{-1})^{C_2}(1-e^{-2t})^{-\sum_{k=1}^l i_l}
    \\
    &\leq C_L\log(\epsilon^{-1})^{C_2}(1-e^{-2t})^{-L}.
\end{align*}

Now, from \eqref{eq:Chmoche} we deduce that
\begin{align*}
\|&\nabla^{i_B}\left(z\mapsto \frac{\nabla^{|B|} Q_{t} r(z)}{Q_{t} r(z)}\right)(x)\|\nonumber\\
    \leq & C_L\log(\epsilon^{-1})^{C_2}(1-e^{-2t})^{-i_B/2}e^{-t(i_B+|B|)}\left(\int \|\frac{\nabla^{|B|}r(y)}{r(y)}-\int \frac{\nabla^{|B|}r(z)}{r(z)}dp^{t,x}(z)\|^2dp^{t,x}(y)\right)^{1/2}.
\end{align*}
Let us now give an estimate on this last term.

\begin{proposition}\label{prop:shouldthelastone}
There exists $C^\star>0$ such that if $t\in (0, \log(\epsilon^{-1})^{-C^\star})$ and $x\in A_t^\epsilon$, then for all $i\in \{1,...,\lfloor \beta \rfloor\}$ we have
    $$\int \|\frac{\nabla^i r}{r}(y)-\frac{\nabla^i r}{r}(e^{-t}x)\|^2dp^{t,x}(y)\leq C\log(\epsilon^{-1})^{C_2(1+i)}(1-e^{-2t})^{\beta-\lfloor \beta \rfloor }.$$
\end{proposition}

The proof of Proposition \ref{prop:shouldthelastone} can be found in Section \ref{sec:prop:shouldthelastone}. Using Proposition \ref{prop:shouldthelastone} we deduce that
\begin{align*}
\|&\nabla^{i_B}\left(z\mapsto \frac{\nabla^{|B|} Q_{t} r(z)}{Q_{t} r(z)}\right)(x)\|\nonumber\\
    \leq & C\log(\epsilon^{-1})^{C_2(1+i_B)}(1-e^{-2t})^{-i_B/2+(\beta-\lfloor \beta \rfloor)/2}e^{-t(i_B+|B|)}.
\end{align*}

Therefore, using \eqref{eq:nablajscoregene} we have
\begin{align*}
     \|\nabla^js(t,x)\|\leq
    & C_j\log(\epsilon^{-1})^{C_2(1+j)}\sum_{\pi \in \Pi_{\lfloor \beta \rfloor}} \sum_{\substack{i_{\pi_1},...,i_{\pi_|B|}\\i_{\pi_1}+...+i_{\pi_{|B|}}=j+1-\lfloor \beta \rfloor}} \prod_{B\in \pi}(1-e^{-2t})^{-i_B/2+(\beta-\lfloor \beta \rfloor)/2}e^{-t(i_B+|B|)}\\
     \leq & C_j\log(\epsilon^{-1})^{C_2(1+j)}e^{-t(j+1)}(1-e^{-2t})^{-\frac{1}{2}(j+1-\lfloor \beta \rfloor-(\beta-\lfloor \beta \rfloor))}\\
     \leq & C_j\log(\epsilon^{-1})^{C_2(1+j)}e^{-t(j+1)}(1-e^{-2t})^{-\frac{1}{2}(j+1-\beta)},
\end{align*}
which concludes the bound on $\nabla^j s(t,\cdot)$ with $j> \lfloor \beta \rfloor -1$ in the case $\beta \geq 1$ as stated in the next Proposition.
\begin{proposition}
 There exists $C^\star>0$ such that for all $t\in (0, \log(\epsilon^{-1})^{-C^\star})$, $x \in A_0^\epsilon$ defined in \eqref{align:aepst} and $j\in \mathbb{N}_{\geq \lfloor \beta \rfloor}$, we have supposing $\beta\geq 1$ that
\begin{equation}
   \|\nabla^{j}\big( s(t,\cdot)+\text{Id}\big)(x)\|\leq  C_j\log(\epsilon^{-1})^{C_2(1+j)}e^{-t(j+1)}(1-e^{-2t})^{-\frac{1}{2}(j+1-\beta)}.
\end{equation}
\end{proposition}

The case \( \beta \in (0,1) \) can be handled in a similar manner by differentiating \eqref{eq:todifferentiatewhenbeta0} and obtaining polynomials centered with respect to \( \varphi^{t,x} \); we omit the details for conciseness, as the proof closely follows the same structure.

\subsection{Conclusion}
Let us summarize the results of the precedent sections.
\begin{itemize}
    \item If $\beta \in (0,1)$, we have for all $j\in \mathbb{N}_{\geq 0}$ and $x\in A_t^\epsilon$, $$\|\nabla^j (z\mapsto s(t,z)+z)(x)\| \leq C_j\log(\epsilon^{-1})^{C_2(1+j)}e^{-(j+1)t} (1-e^{-2t})^{-\frac{1}{2}(j+1-\beta)}.$$
        \item If $\beta \geq 1$, we have for all $x\in A_t^\epsilon$ and $j\in \{0,\lfloor \beta \rfloor -1\}$,
        $$\|\nabla^j (z\mapsto s(t,z)+z)(x)\| \leq C\log(\epsilon^{-1})^{C_2}e^{-(j+1)t},$$
        and if $j\in \mathbb{N}_{\geq \lfloor \beta \rfloor}$, we have
        $$\|\nabla^j (z\mapsto s(t,z)+z)(x)\| \leq C_j\log(\epsilon^{-1})^{C_2(1+j)}e^{-(j+1)t} (1-e^{-2t})^{-\frac{1}{2}(j+1-\beta)}.$$
\end{itemize}
This concludes the bound on the $\mathcal{H}^j$ norm of $s(t,\cdot)$ when $j$ is an integer. Following the method developed in Section~B.2 of \cite{stephanovitch2024smooth}, one can show that for any integer \( j \geq 1 \) and \( \eta \in (0,1) \), the following estimate holds:
\[
\left\| \nabla^j \left( z \mapsto s(t,z) + z \right) \right\|_{\mathcal{H}^{\eta}(A_t^\epsilon)} \leq C_{j,\eta} \log(\epsilon^{-1})^{C_2(1 + j+\eta)} e^{-(1+j + \eta)t} \left(1 + t^{-\frac{1}{2} \big( (1 + j + \eta - \beta) \vee 0 \big)} \right),
\]
where \( A_t^\epsilon \) is defined as previously. Since the proof closely mirrors that of \cite{stephanovitch2024smooth}, we omit it here for conciseness.

\subsection{Details of the proofs of higher-order results and neural net approximation}
\subsubsection{Proof of Proposition \ref{prop:masstlarge}}\label{sec:prop:masstlarge}
\begin{proof}
We have for \( X_0 \sim p \) and \( Z \sim \mathcal{N}(0, I_d) \) independent random variables,
\begin{align}
    p_t(A^\epsilon_0)=&\mathbb{P}(e^{-t}X_0+\sqrt{1-e^{-2t}}Z\in A^\epsilon_0)\nonumber\\
    \geq &\mathbb{P}\Big( \|X_0-y^\star\|^2\leq \frac{C^\star}{4} \big(\log(\epsilon^{-1})(1+\alpha^{-1})+K\big)\Big)\nonumber\\
    & \times \mathbb{P}\Big(\|Z-y^\star\|^2\leq \frac{C^\star}{4} \big(\log(\epsilon^{-1})(1+\alpha^{-1})+K\big)\Big)\nonumber\\
\geq & 1-\epsilon
\end{align}
for $C^\star>0$ large enough as by assumption we have $\|y^\star\|\leq K$.
\end{proof}

\subsubsection{Proof of Proposition \ref{prop:masstsmall}}\label{sec:prop:masstsmall}
\begin{proof}
For \( X_0 \sim p \) and \( Z \sim \mathcal{N}(0, I_d) \) independent random variables, we have
\begin{align}\label{eq:fhosijsbdcbjksds}
    p_t(A^\epsilon_0)=&\mathbb{P}(e^{-t}X_0+\sqrt{1-e^{-2t}}Z\in A^\epsilon_0)\nonumber\\
    \geq & \mathbb{P}(X\in \{u(y)\leq \log(\epsilon^{-1})+R\})\mathbb{P}(\sqrt{1-e^{-2t}}\|Z\|\leq K^{-1} (\log(\epsilon^{-1})+R)^{-K})\nonumber\\
    \geq & 1-\epsilon,
\end{align}
for $R>0$ and $C^\star>>R$ large enough recalling Assumption \ref{assum:higherorder}-3 and using that $\|a\|_\infty \leq K$.
\end{proof}

\subsubsection{Proof of Lemma \ref{lemma:concentrationsmallthigh}}\label{sec:lemma:concentrationsmallthigh}

\begin{proof}

Let $y^\star$ defined in \eqref{eq:ystar},
we have
\begin{align}\label{alignjhdskaeoedjhe}
\int \|F_{l} (y-e^{-t}x)\| dp^{t,x}(y)
    & \leq C \int \sum_{i=1}^{l}(\|y\|^i+\|e^{-t}x\|^i)dp^{t,x}(y)\nonumber\\
    & \leq C \log(\epsilon^{-1})^{Cl } + C\int \sum_{i=1}^{l}\|y\|^idp^{t,x}(y).
\end{align}
Now, as on one hand we have $\|y^\star\|\leq K$ by Assumption \ref{assum:higherorder}-2 and on the other hand we have $$\int \mathds{1}_{y\notin B\big(y^\star,\log(\epsilon^{-1})^{1/2}\big)}p(y)dy\leq C\epsilon^\alpha$$
by Assumption \ref{assum:1}-1,
we deduce that there exists $R>0$ independent of $p$ such that,
$$\int \mathds{1}_{y\in B(0,R)}p(y)dy\geq 1/2.$$
Therefore, for $x\in A_t^\epsilon$ we have 
\begin{align}\label{eq:hbckzndihhdiidfiebks}
    \int r(z)\varphi^{t,x}(z)dz\geq& \int r(z)\varphi^{t,x}(z) \mathds{1}_{\|y\|\leq R}dz\nonumber\\
    \geq & C^{-1}(1-e^{-2t})^{-d/2} \exp\left(-\frac{C(\|e^{-t}x\|^2+R^2)}{1-e^{-2t}}\right)\nonumber\\
    \geq & C^{-1} \exp\left(-\frac{C(C^\star \log(\epsilon^{-1})+C+R^2)}{1-e^{-2t}}\right).
\end{align}
Furthermore, for $i\in \{1,...,l\}$ and $R_2>0$, by Assumption \ref{assum:1}-1 we have
\begin{align}\label{eq:hbckzndihfiebks}
\int & \|y\|^ir(y)\varphi^{t,x}(y)\mathds{1}_{\|y\|\geq \log(\epsilon^{-1})^{R_2}}dy\nonumber\\
= & r(y^\star) \int \|y\|^ir(y)r(y^\star)^{-1}\varphi^{t,x}(y)\mathds{1}_{\|y\|\geq \log(\epsilon^{-1})^{R_2}}dy\nonumber\\
\leq & C \int (\|y-e^{-t}x\|^i+\|e^{-t}x\|^i) \exp\left(-\frac{\alpha\log(\epsilon^{-1})^{2R_2}}{2(1-e^{-2t})}-\frac{\|y-e^{-t}x\|^2}{2(1-e^{-2t})}\right)dy\nonumber\\
\leq & C\log(\epsilon^{-1})^{C_2} \int \exp\left(-\frac{\alpha\log(\epsilon^{-1})^{R_2}}{2(1-e^{-2t})}-\frac{\|y-e^{-t}x\|^2}{4(1-e^{-2t})}\right)dy\nonumber\\
\leq & C\exp\left(-\frac{\alpha\log(\epsilon^{-1})^{R_2}}{4(1-e^{-2t})}\right).
\end{align}

Then, putting together \eqref{eq:hbckzndihhdiidfiebks} and \eqref{eq:hbckzndihfiebks} we deduce that
\begin{align*}
    \int \sum_{i=1}^{l}\|y\|^idp^{t,x}(y)\leq & C \log(\epsilon^{-1})^{lR_2}+ \sum_{i=1}^{l}\int \|y\|^i\frac{r(y)\varphi^{t,x}(y)}{\int r(z)\varphi^{t,x}(z)dz}\mathds{1}_{\|y\|\geq \log(\epsilon^{-1})^{R_2}}dy\nonumber\\
    \leq & C \log(\epsilon^{-1})^{lR_2 }+ C \exp\left(\frac{C(C^\star \log(\epsilon^{-1})+C+R^2)}{1-e^{-2t}}\right)\sum_{i=1}^{l}\exp\left(-\frac{\alpha\log(\epsilon^{-1})^{R_2}}{4(1-e^{-2t})}\right)\nonumber\\
    \leq &C \log(\epsilon^{-1})^{lR_2 },
\end{align*}
taking $R_2>0$ large enough. Putting this estimate in the bound \eqref{alignjhdskaeoedjhe}, we get the result.

\end{proof}

\subsubsection{Proof of Lemma \ref{lemma:boundsquared}}\label{sec:lemma:boundsquared}

\begin{proof}
We have
\begin{align*}
   |r(y)-r(e^{-t}x)|=&|e^{-u(y)+a(y)+\|y\|^2/2}- e^{-u(e^{-t}x)+a(e^{-t}x)+\|e^{-t}x\|^2/2}|\\
    =&|e^{a(e^{-t}x)}(e^{-u(e^{-t}x)+\|e^{-t}x\|^2/2}- e^{-u(y)+\|y\|^2/2})+ e^{-u(y)+\|y\|^2/2}(e^{a(e^{-t}x)}-e^{a(e^{-t}x)})|\\
    \leq& C(r(e^{-t}x)\vee r(y)) \int_0^1\|-\nabla u (y+s(e^{-t}x-y))+y+s(e^{-t}x-y)\|\| e^{-t}x-y\|ds\\
    &+ Ce^{-u(y)+\|y\|^2/2}\| e^{-t}x-y\|^{\beta\wedge 1}.
\end{align*}
Taking 
\begin{equation}\label{eq:delta1}
\delta=R^{-1}\log(\epsilon^{-1})^{-R},
\end{equation}
with $R>0$ large enough,
we provide separate bounds whether $\|y-e^{-t}x\|\leq \delta$ or not.
We have
\begin{align*}
    \int & \left( \frac{r(y)-r(e^{-t}x)}{Q_tr(x)}\right)^{2}\varphi^{t,x}(y)\mathds{1}_{\|y-e^{-t}x\|\leq \delta}dy\\
    \leq &\frac{C}{Q_tr(x)^2} \int (r(e^{-t}x)\vee r(y))^2 \|\sup_{s\in[0,1]}-\nabla u (y+s(e^{-t}x-y))+y+s(e^{-t}x-y)\|^2 \| e^{-t}x-y\|^2\\
    & \times \varphi^{t,x}(y)\mathds{1}_{\|y-e^{-t}x\|\leq \delta}dy\\
    &+\frac{C}{Q_tr(x)^2} \int r(y)^2 \| e^{-t}x-y\|^{2(\beta\wedge 1)} \varphi^{t,x}(y)\mathds{1}_{\|y-e^{-t}x\|\leq \delta}dy.
\end{align*} 
Taking $R>0$ large enough  and using Proposition \ref{prop:enfaitelemma}, we have that for all $y\in B(e^{-t}x,\delta)$,
 $$\frac{(r(e^{-t}x)\vee r(y))^2}{Q_tr(x)^2}\leq C,$$ 
 and using that $e^{-t}x\in A_t^\epsilon$, we have from Assumptions \ref{assum:higherorder} that
$$\|\sup_{s\in[0,1]}-\nabla u (y+s(e^{-t}x-y))+y+s(e^{-t}x-y)\|\leq C\log(\epsilon^{-1})^{C_2}.$$ 
Therefore,
 we obtain
\begin{align*}
   & \frac{1}{Q_tr(x)^2} \int (r(e^{-t}x)\vee r(y))^2 \|\sup_{s\in[0,1]}-\nabla u (y+s(e^{-t}x-y))+y+s(e^{-t}x-y)\|^2 \| e^{-t}x-y\|^2 \varphi^{t,x}(y)\mathds{1}_{\|y-e^{-t}x\|\leq \delta}dy\\
    &+\frac{C}{Q_tr(x)^2} \int r(y)^2 \| e^{-t}x-y\|^{2(\beta\wedge 1)} \varphi^{t,x}(y)\mathds{1}_{\|y-e^{-t}x\|\leq \delta}dy\\
    & \leq C\log(\epsilon^{-1})^{C_2}\int \|y-e^{-t}x\|^{2(\beta\wedge 1)}\mathds{1}_{\|y-e^{-t}x\|\leq \delta}\varphi^{t,x}(y)dy\\
    & \leq C\log(\epsilon^{-1})^{C_2} (1-e^{-2t})^{2(\beta\wedge 1)}.
\end{align*}
On the other hand, using Proposition \ref{prop:enfaitelemma} we have
\begin{align*}
     \int  \left( \frac{r(y)-r(e^{-t}x)}{Q_tr(x)}\right)^{2}\varphi^{t,x}(y)\mathds{1}_{\|y-e^{-t}x\|\geq \delta}dy
    &\leq C\frac{\exp(-\frac{\delta^2}{2(1-e^{-2t})})}{Q_tr(x)^2(1-e^{-2t})^{d/2}}\\
    & \leq C\sup_{w\in B(e^{-t}x,(1-e^{-2t})^{1/2})}\exp\left(2u(w)\right)\exp(-\frac{\delta^2}{4(1-e^{-2t})})\\
    & \leq C \exp(-\frac{\delta^2}{4(1-e^{-2t})})\\
    & \leq C(1-e^{-2t}),
\end{align*}
 taking $C^\star>0$ large enough compared to $R$.

\end{proof}

\subsubsection{Proof of Proposition \ref{prop:enfaiteccestcoro}}\label{sec:prop:enfaiteccestcoro}
\begin{proof}
Using Proposition \ref{prop:shouldthelastone} and Assumption \ref{assum:higherorder}-3 we have
\begin{align*}
    \|\int \frac{\nabla^ir(y)}{r(y)} dp^{t,x}(y)dy\|\leq & \|\int \frac{\nabla^ir(e^{-t}x)}{r(e^{-t}x)} dp^{t,x}(y)dy\|+\|\int\left(\frac{\nabla^ir(e^{-t}x)}{r(e^{-t}x)}- \frac{\nabla^ir(y)}{r(y)}\right) dp^{t,x}(y)dy\|\\
    \leq & C\log(\epsilon^{-1})^{C_2}.
\end{align*}
\end{proof}

\subsubsection{Proof of Proposition \ref{prop:thepoly}}\label{sec:prop:thepoly}

\begin{proof}
From Lemma 4 in \cite{stephanovitch2024smooth}, we have
     for all $f\in L^2(\mathbb{R}^d)$ that
    \begin{align*}
    \int f(y)\nabla_x p^{t,x}(y)d(y) =\int f(y)\otimes \frac{y-\int zdp^{t,x}(z)}{1-e^{-2t}} dp^{t,x}(y),
\end{align*}

so we deduce that the result is true for $j=1$.
By induction suppose that \eqref{align:nablajqsurq} is true for a certain $j\geq 1$, then
\begin{align*}
    &\nabla^{j+1}\left(z\mapsto \frac{\nabla^{l} Q_{t} r(z)}{Q_{t} r(z)}\right)(x)\\
    =&e^{-tj}\int \frac{\nabla^{l}r(y)}{r(y)}\otimes\left(\nabla_x P^{j}(y,x)-\nabla_x\int  P^{j}(z,x)dp^{t,x}(z) \right)dp^{t,x}(y)\\
    &+ e^{-tj}\int \frac{\nabla^{l}r(y)}{r(y)}\otimes\left( P^{j}(y,x)-\int P^{j}(z,x)dp^{t,x}(z) \right)\nabla_x p^{t,x}(y)dy\\
    =&e^{-tj}\int \frac{\nabla^{l}r(y)}{r(y)}\otimes\left(\nabla_x P^{j}(y,x)-\int  \nabla_xP^{j}(z,x)dp^{t,x}(z) \right)dp^{t,x}(y)\\
    &+e^{-t(j+1)}\int \frac{\nabla^{l}r(y)}{r(y)}\otimes\int  \left( P^{j}(z,x)-\int P^{j}(w,x)dp^{t,x}(w) \right)\otimes \frac{z-\int wdp^{t,x}(w)}{1-e^{-2t}} dp^{t,x}(z) dp^{t,x}(y)\\
    &+ e^{-t(j+1)}\int \frac{\nabla^{l}r(y)}{r(y)}\otimes\left( P^{j}(y,x)-\int P^{j}(z,x)dp^{t,x}(z) \right)\otimes \frac{y-\int zdp^{t,x}(z)}{1-e^{-2t}} dp^{t,x}(y)dy\\
    =&e^{-t(j+1)}\int \frac{\nabla^{l}r(y)}{r(y)}\otimes\left( P^{j+1}(y,x)-\int  P^{j+1}(z,x)dp^{t,x}(z) \right)dp^{t,x}(y)
\end{align*}
with 
\begin{align*}
    P^{j+1}(y,x) = e^{t}\nabla_x P^{j+1}(y,x)+\left( P^{j}(y,x)-\int P^{j}(z,x)dp^{t,x}(z) \right)\otimes \frac{y-\int zdp^{t,x}(z)}{1-e^{-2t}} .
\end{align*}

\end{proof}

\subsubsection{Proof of Lemma \ref{lemma:concetrationaroudx}}\label{sec:lemma:concetrationaroudx}

Let us first give an upper bound on $Q_tr(x)$.

\begin{proposition}\label{prop:enfaitelemma}
 There exists $C^\star>0$ such that if $t\leq \log(\epsilon^{-1})^{-C^\star}$ and $x\in A_t^\epsilon$, then  
 \begin{align*}
    \frac{1}{Q_tr(x)}\leq C\sup_{w\in B(e^{-t}x,(1-e^{-2t})^{1/2})}\exp\left(u(w)\right),
\end{align*}
and for all $\gamma \leq |Ku(e^{-t}x)|^{-(K+ 1)}\wedge 1,y\in B(e^{-t}x,\gamma)$, we have
\begin{align*}
  \frac{r(y)}{Q_tr(x)}  \leq C.
\end{align*}
\end{proposition}

\begin{proof}
We have \begin{align*}
    Q_{t}r(x) & = \int \varphi^{t,x}(y)r(y)dy\nonumber\\
    & \geq  \min_{w\in \overline{B}(e^{-t}x,(1-e^{-2t})^{1/2}))}\exp\left(-u(w)+a(w)\right)\int \varphi^{t,x}(y)\mathds{1}_{\|y-e^{-t}x\|\leq(1-e^{-2t})^{1/2}}dy\nonumber\\
    & \geq C^{-1}\min_{w\in \overline{B}(e^{-t}x,(1-e^{-2t})^{1/2}))}\exp\left(-u(w)\right)\int \mathds{1}_{\|y\|\leq\frac{(1-e^{-2t})^{1/2}}{(1-e^{-2t})^{1/2}}}\gamma_d(y)dy,\nonumber\\
    & \geq C^{-1}\min_{w\in \overline{B}(e^{-t}x,(1-e^{-2t})^{1/2}))}\exp\left(-u(w)\right).
\end{align*}
Furthermore, for $\gamma \leq |Ku(e^{-t}x)|^{-(K+ 1)}\wedge 1$ and $y\in B(e^{-t}x,\gamma)$, we have by Assumption \ref{assum:higherorder}-3 that
\begin{align*}
    \frac{r(y)}{Q_tr(x)}  \leq & C \exp\left(\max_{z\in \overline{B}(e^{-t}x,\gamma)}u(z)-\min_{w\in \overline{B}(e^{-t}x,\gamma)}u(w)\right)\\
    \leq & C\exp\left(\max_{z\in \overline{B}(e^{-t}x,\gamma)}\|\nabla u(z)\|\gamma\right)
    \\
    \leq & \exp\left(K(|u(e^{-t}x)|^K+1)\gamma\right)\\
    \leq & C.
\end{align*}
\end{proof}
We can now give the proof of the Lemma \ref{lemma:concetrationaroudx}.
\begin{proof}[Proof of Lemma \ref{lemma:concetrationaroudx}]

Let $\phi:\mathbb{R}^d\rightarrow \mathbb{R}^d$ be the $C^1$ diffeomorphism of the polar change of coordinates, we have 
\begin{align*}
    \int&\|y-e^{-t}x\|^{2\eta}dp^{t,x}(y)\\
     = &\int_0^\infty\int_0^\pi... \int_0^{2\pi} \rho^{2\eta }\frac{r(e^{-t}x +\phi(\rho,\theta_1,...,\theta_{d-1}))}{Qtr(x)} \frac{e^{-\frac{\rho^2}{2(1-e^{-2t})}}}{(2\pi(1-e^{-2t}))^{d/2}}|\text{det}(\nabla \phi(\rho,\theta_1,...,\theta_{d-1}))|d\rho d\theta_1...d\theta_{d-1}\\
    \leq & \frac{C}{Qtr(x)}\int_0^{\infty}\int_0^\pi... \int_0^{2\pi} \rho^{2\eta +d-1 }r(e^{-t}x +\phi(\rho,\theta_1,...,\theta_{d-1}))\frac{e^{-\frac{\rho^2}{2(1-e^{-2t})}}}{(1-e^{-2t})^{d/2}}d\rho d\theta_1...d\theta_{d-1}\\
    \leq & \frac{(1-e^{-2t})^{\eta}}{Qtr(x)}\int_0^{\infty}\int_0^\pi... \int_0^{2\pi} \rho^{2\eta +d-1 }r(e^{-t}x +\phi(\rho(1-e^{-2t})^{1/2},\theta_1,...,\theta_{d-1}))e^{-\frac{\rho^2}{2}}d\rho d\theta_1...d\theta_{d-1}.
\end{align*}
Let $\gamma=(1-e^{-2t})^{-1/4}$, using Proposition \ref{prop:enfaitelemma} we have
\begin{align*}
  &\frac{1}{Q_tr(x)} \int_0^{\infty}\int_0^\pi... \int_0^{2\pi} \rho^{2\eta +d-1 }r(e^{-t}x +\phi(\rho(1-e^{-2t})^{1/2},\theta_1,...,\theta_{d-1}))e^{-\frac{\rho^2}{2}}\mathds{1}_{\{\rho\leq \gamma\}}d\rho d\theta_1...d\theta_{d-1}  \\
  &\leq C\int_0^{\infty}\int_0^\pi... \int_0^{2\pi} \rho^{2\eta +d-1 }\exp\left(\max_{z\in \overline{B}(e^{-t}x,\gamma(1-e^{-2t})^{1/2})}u(z)-\min_{w\in \overline{B}(e^{-t}x,\gamma(1-e^{-2t})^{1/2})}u(w)\right)\\
  &\quad \quad \times e^{-\frac{\rho^2}{2}}\mathds{1}_{\{\rho\leq \gamma\}}d\rho d\theta_1...d\theta_{d-1}
  \\
  &\leq C\int_0^{\infty}\int_0^\pi... \int_0^{2\pi} \rho^{2\eta +d-1 }\exp\left(\max_{z\in \overline{B}(e^{-t}x,(1-e^{-2t})^{1/4})}\|\nabla u(z)\|C(1-e^{-2t})^{1/4}\right)e^{-\frac{\rho^2}{2}}\mathds{1}_{\{\rho\leq \gamma\}}d\rho d\theta_1...d\theta_{d-1}
  \\
  &\leq C\int_0^{\infty}\int_0^\pi... \int_0^{2\pi} \rho^{2\eta +d-1 }\exp\left(C(\| u(e^{-t}x)\|+1)^K(1-e^{-2t})^{1/4}\right)e^{-\frac{\rho^2}{2}}\mathds{1}_{\{\rho\leq \gamma\}}d\rho d\theta_1...d\theta_{d-1}
  \\
  &\leq C\int_0^{\infty}\int_0^\pi... \int_0^{2\pi} \rho^{2\eta +d-1 }e^{-\frac{\rho^2}{2}}\mathds{1}_{\{\rho\leq \gamma\}}d\rho d\theta_1...d\theta_{d-1}\\
  &\leq C_\eta.
\end{align*}
The same way, using again Proposition \ref{prop:enfaitelemma} we obtain
\begin{align*}
    &\frac{1}{Q_tr(x)} \int_0^{\infty}\int_0^\pi... \int_0^{2\pi} \rho^{2\eta +d-1 }r(e^{-t}x +\phi(\rho(1-e^{-2t})^{1/2},\theta_1,...,\theta_{d-1}))e^{-\frac{\rho^2}{2}}\mathds{1}_{\{\rho> \gamma\}}d\rho d\theta_1...d\theta_{d-1}\\
    &\leq C_\eta\frac{e^{-\frac{\gamma^2}{4}}}{Q_tr(x)}\\
    & \leq C_\eta\exp\left(-\frac{(1-e^{-2t})^{-1/2}}{4}\right)\exp\left(u(e^{-t}x)+\max_{z\in \overline{B}(e^{-t}x,(1-e^{-2t})^{1/2})}\|\nabla u(z)\|(1-e^{-2t})^{1/2}\right)\\
     & \leq C_\eta\exp\left(-\frac{(1-e^{-2t})^{-1/2}}{8}\right),
\end{align*}
taking $C^\star>0$ large enough.
\end{proof}

\subsubsection{Proof of Proposition \ref{prop:shouldthelastone}}\label{sec:prop:shouldthelastone}

\begin{proof}As $a\in \mathcal{H}^\beta_K$ and $u\in C^{\lfloor \beta \rfloor+1}$, we have
\begin{align*}
    \|\frac{\nabla^i r}{r}(y)-\frac{\nabla^i r}{r}(e^{-t}x)\|\leq & C \|y-e^{-t}x\|^{\beta-\lfloor \beta \rfloor}\left(1+\sum_{j=1}^{\lfloor \beta \rfloor+1}\sup_{t\in [0,1]}\|\nabla^j u(y+t(e^{-t}x-y))\|\right)
\end{align*}
so  $x\in A_\epsilon^t$ implies that for all $y\in \mathbb{R}^d$ such that $\|e^{-t}x-y\|\leq R^{-1}\log(\epsilon^{-1})^{-R}$ for $R>0$ large enough, we have
\begin{align*}
    \|\frac{\nabla^i r}{r}(y)-\frac{\nabla^i r}{r}(e^{-t}x)\|\leq & C\log(\epsilon^{-1})^{C_2(1+i)} \|y-e^{-t}x\|^{\beta-\lfloor \beta \rfloor}.
\end{align*}
Taking $\delta=R^{-1}\log(\epsilon^{-1})^{-R}$,
we provide separate bounds on 
    \begin{align*}
     \int \|\frac{\nabla^i r}{r}(y)-\frac{\nabla^i r}{r}(e^{-t}x)\|^2\mathds{1}_{\|y-e^{-t}x\|\leq \delta}dp^{t,x}(y)        
\end{align*}   
and
\begin{align*}  
     \int \|\frac{\nabla^i r}{r}(y)-\frac{\nabla^i r}{r}(e^{-t}x)\|^2\mathds{1}_{\|y-e^{-t}x\|\geq \delta}dp^{t,x}(y).
\end{align*}
For the first term,  we have
\begin{align*}
    \int & \|\frac{\nabla^i r}{r}(y)-\frac{\nabla^i r}{r}(e^{-t}x)\|^2\mathds{1}_{\|y-e^{-t}x\|\leq \delta}dp^{t,x}(y)\\
    & \leq C\log(\epsilon^{-1})^{C_2(1+i)}\int \|y-e^{-t}x\|^{2(\beta-\lfloor \beta \rfloor)}\mathds{1}_{\|y-e^{-t}x\|\leq \delta}dp^{t,x}(y).
\end{align*}
so using Lemma \ref{lemma:concetrationaroudx} we get
\begin{align*}
    \int \|\frac{\nabla^i r}{r}(y)-\frac{\nabla^i r}{r}(e^{-t}x)\|^2\mathds{1}_{\|y-e^{-t}x\|\leq \delta}dp^{t,x}(y) & \leq C\log(\epsilon^{-1})^{C_2(1+i)}(1-e^{-2t})^{2(\beta-\lfloor \beta \rfloor)}.
\end{align*}

Now, for the other term we have
\begin{align*}
    &\int \|\frac{\nabla^i r}{r}(y)-\frac{\nabla^i r}{r}(e^{-t}x)\|^2\mathds{1}_{\|y-e^{-t}x\|\geq \delta}dp^{t,x}(y)\\
   & \leq  C\int \|\frac{\nabla^i r}{r}(y)-\frac{\nabla^i r}{r}(e^{-t}x)\|^2\frac{\exp(-u(y))}{Qtr(x)(1-e^{-2t})^{d/2}}\exp(-\frac{\delta^2}{2(1-e^{-2t})})dy\\
    &\leq C\log(\epsilon^{-1})^{C_2(1+i)}\frac{\exp(-\frac{\delta^2}{2(1-e^{-2t})})}{Q_tr(x)(1-e^{-2t})^{d/2}}\int \exp(-u(y)/2)dy\\
    &\leq C\log(\epsilon^{-1})^{C_2(1+i)}\frac{\exp(-\frac{\delta^2}{2(1-e^{-2t})})}{Q_tr(x)(1-e^{-2t})^{d/2}}\\
    & \leq C\exp(-\frac{1}{4(1-e^{-2t})^{1/2}}),
\end{align*}
using Proposition \ref{prop:enfaitelemma} and taking $C^\star>0$ large enough.

\end{proof}

\subsubsection{Proof of Corollary \ref{coro:diffusregutime}}\label{sec:timeregu}

\begin{proof}
Suppose that $\gamma$ is an integer, the non-integer case can be treated similarly. Let us first recall from \eqref{eq:jnhodishsiozkdp} that
\begin{equation*}
\nabla \log p_t(x) = -x + \nabla \log Q_t r(x).
\end{equation*}
Now, we have 
\begin{align*}
\partial^{\gamma}_t \nabla \log Q_tr(x)
    & = \nabla \sum_{\pi \in \Pi_{\gamma}} (-1)^{|\pi|+1}\prod_{B\in \pi} \frac{\partial_t^{|B|} Q_tr(x)}{Q_tr(x)}\\
    & = \sum_{\pi \in \Pi_{\gamma}} (-1)^{|\pi|+1} \sum_{B\in \pi}\frac{\nabla \partial_t^{|B|} Q_tr(x)}{Q_tr(x)} \prod_{C\neq B\in \pi} \frac{ \partial_t^{|C|} Q_tr(x)}{Q_tr(x)}.
\end{align*}
Now, as
\begin{align*}
    \partial_t Q_tr(x) = \Delta Q_tr(x) -\langle x, \nabla Q_tr(x)\rangle,
\end{align*}

we have for all $\pi \in \Pi_{\gamma}$ that
$$\sum_{B\in \pi}\frac{\nabla \partial_t^{|B|} Q_tr(x)}{Q_tr(x)} \prod_{C\neq B\in \pi} \frac{ \partial_t^{|C|} Q_tr(x)}{Q_tr(x)}= \sum_{B\in \pi}\frac{\nabla F^{|B|} Q_tr(x)}{Q_tr(x)} \prod_{C\neq B\in \pi} \frac{ F^{|C|} Q_tr(x)}{Q_tr(x)}$$
with 
$$F^i=F\circ F^{i-1}, \quad F(g) = \Delta g - \langle x,\nabla g\rangle.$$
In terms of space-regularity, we know from Section \ref{sec:theo:highregu} that 
$$\|\frac{\nabla^i Q_tr(x)}{Q_tr(x)}\|\leq C\log(\epsilon^{-1})^{C_2} e^{-t}\text{ for } i\in \{1,...,\lfloor \beta\rfloor\}$$
and
$$\|\frac{\nabla^{\lfloor \beta\rfloor+j} Q_tr(x)}{Q_tr(x)}\|\leq C_j\log(\epsilon^{-1})^{C_2(1+j)}e^{-t}(1+t^{-\frac{j-(\beta-\lfloor \beta\rfloor)}{2}}) \text{ for } j\in \mathbb{N}_{>0}.$$
We deduce that for all $\pi \in \Pi_\gamma$ we have
\begin{align*}
\|&\sum_{B\in \pi}\frac{\nabla F^{|B|} Q_tr(x)}{Q_tr(x)} \prod_{C\neq B\in \pi} \frac{ F^C Q_tr(x)}{Q_tr(x)}\|\\
&\leq C_j\log(\epsilon^{-1})^{C_2(1+\gamma)}e^{-t}\sum_{B\in \pi}(1+t^{-(|B|-(\beta-1)/2)}\mathds{1}_{|B|\geq (\beta-1)/2})(1+\|x\|^{|B|})\\
& \times \prod_{C\neq B\in \pi} (1+t^{-(|C|-\beta/2)}\mathds{1}_{|C|\geq \beta/2})(1+\|x\|^{|C|}) \\
& \leq C_j\log(\epsilon^{-1})^{C_2(1+\gamma)}e^{-t}(1+t^{-(\gamma-(\beta-1)/2)}\mathds{1}_{\gamma\geq (\beta-1)/2})(1+\|x\|^{\gamma}).
\end{align*}
Recalling the definitions of $A_t^\epsilon$ \eqref{eq:aepstsmall} and \eqref{eq:fhosijsbdcbjksds}, we have that for $x\in A_t^\epsilon$
$$\|x\|\leq C\log(\epsilon^{-1})^{C_2}.$$
Therefore, we conclude that for $x\in A_t^\epsilon$,
\begin{equation*}
\|s(\cdot,x)+x\|_{ \mathcal{H}^{\gamma}([t,\infty))}\leq C_\gamma\log(\epsilon^{-1})^{C_2(1+\gamma)}e^{-t}\left(1+t^{-\big((\frac{1}{2}+\gamma-\frac{\beta}{2})\vee 0}\big)\right).
\end{equation*}  
The same way, we obtain
that 
\begin{equation*}\|\partial_t^k \big(s(t,\cdot)+\text{Id}(\cdot)\big)\|_{\mathcal{H}^{\gamma}(A_t^\epsilon)}\leq C_{k,\gamma}\log(\epsilon^{-1})^{C_2(1+k+\gamma)}e^{-t}\left(1+t^{-\big((\frac{1}{2}+k+\frac{\gamma-\beta}{2})\vee 0\big)}\right).
\end{equation*}
\end{proof}

\subsubsection{Proof of Corollary \ref{coro:approxnn}}\label{sec:coro:approxnn}

\begin{proof}

Let us divide the interval $[\epsilon^{\theta},\epsilon^2]$ into $[t_0,t_1], [t_1,t_2],...[t_{m-1},t_m]$ with $t_0=\epsilon^{\theta}$, $t_{i+1}=2t_i$ for $i\in \{0,...,m-2\}$ and $t_m=\epsilon^2.$ In particular we have $m\leq C_\theta \log(\epsilon^{-1})$. We are going to approximate $s(t_i,\cdot)$ and its time-derivatives for all $i\in \{0,...,m\}$ and then use a Taylor approximation in time to define a neural network on the whole interval $[t_i,t_{i+1}]$. Finally we define a neural network on the whole space $[\epsilon^\theta,\epsilon^2]\times \mathbb{R}^d$ by interpolating between the different approximation on each sub-interval.

\textbf{Approximation with polynomials of neural nets}\\
Taking $$\nu = \epsilon^{2(\beta+\theta)},$$
recall from \eqref{eq:hgfjkzudijdjhdhdhhdhd} that for all $t\leq \epsilon^{\theta}$ we have 
$$A_t^\nu=A_{\epsilon^2}^\nu.$$
Let us fix $R>0$ such that $A_{\epsilon^2}^\nu\subset B(0,\log(\epsilon^{-1})^R)$. For $i\in \{0,...,m\}$ and $k\in \{0,...,\lfloor \beta/2 \rfloor\}$, let $\mathcal{F}_i^{k}$ be the class of neural networks (given by either Theorem 5.1 in \cite{de2021approximation} or Theorem 5 in \cite{schmidt2020nonparametric}) with an architecture of depth $O(1)$, width $O(\epsilon^{-d})$ and weights bounded by $O(\epsilon^{-C_2})$ which satisfies that for all $f\in \mathcal{H}^{\beta-2k}_{C\log(\epsilon^{-1})^{C_2}t_i^{-\frac{1}{2}}}(B(0,\log(\epsilon^{-1})^R),\mathbb{R}^{d})$, there exists $N_f\in \mathcal{F}_{i}^k$ satisfying
$$\sup_{x\in B(0,\log(\epsilon^{-1})^R)}\|f-N_f\|\leq  C\log(\epsilon^{-1})^{C_2}\frac{\epsilon^{\beta-2k}}{t_i^{\frac{1}{2}}}.$$
Then from the space-regularity of the derivatives of the score \eqref{eq:reguscorespace}, we have that there exists $N_i^k\in \mathcal{F}_i^k$ such that 
\begin{equation}
    \sup_{x\in A^{\nu}_{\epsilon^2}} \|N_i^k(x)-\partial_t^k s(t_i,x)\|\leq C \log(\epsilon^{-1})^{C_2} \frac{\epsilon^{\beta-2k}}{t_i^{\frac{1}{2}}}.
\end{equation}
For $l\geq 0$, we define $P^{\lfloor l\rfloor}$ the $\lfloor l\rfloor$-th Taylor approximation which given an input $(a_k)_{k=0}^{\lfloor l\rfloor}$, outputs the polynomial $$P^{\lfloor l\rfloor}\big((a_k)_{k=0}^{\lfloor l\rfloor}\big)(t) = \sum_{k=0}^{\lfloor l\rfloor} a_k\frac{(t-t_i)^k}{k!}.$$ 
In particular, for a function $f\in \mathcal{H}^{l}_B$ we have
$$|f(t)-P^{\lfloor l\rfloor}\big((f^{(k)}(t_i))_{k=0}^{\lfloor l\rfloor}\big)(t)|\leq B|t-t_i|^{l}.$$
Now, for $t\in [t_i/2,2t_{i+1}]$, we define the Taylor approximation with respect to time of $s(\cdot,x)$ by
\begin{align}
    s_i(t,x) = P^{\lfloor \frac{\beta+1}{2}\rfloor}\big((N_i^k(x))_{k=0}^{\lfloor \frac{\beta+1}{2}\rfloor} \big)(t).
\end{align}
Then, from the time-regularity of the score \eqref{eq:reguscoretime} with $\gamma=\beta/2$, we have
\begin{align*}
       \|s(t,x)-s_i(t,x)\|  \leq & \|s(t,x)-P^{\lfloor \frac{\beta+1}{2}\rfloor}\big(\partial_t^k s(t_i,x))_{k=0}^{\lfloor \frac{\beta+1}{2}\rfloor }\big)(t)\|\\
        &+ \|P^{\lfloor \frac{\beta+1}{2}\rfloor}\big(\partial_t^k s(t_i,x))_{k=0}^{\lfloor \frac{\beta+1}{2}\rfloor }\big)(t)-s_i(t,x)\|\\
    \leq & C\log(\epsilon^{-1})^{C_2}\left(\frac{(t_{i+1}-t_i)^{\beta/2}}{t_i^{\frac{1}{2}}}+ \sum_{k=0}^{\lfloor \frac{\beta+1}{2}\rfloor}t_i^{-1/2}\epsilon^{\beta-2k}(t_{i+1}-t_i)^{k}\right)
    \\
    \leq & C\log(\epsilon^{-1})^{C_2} \frac{\epsilon^{\beta}}{t^{\frac{1}{2}}}.
\end{align*}
\textbf{Construction of the final neural net}\\
We know from Corollary
3.7 in \cite{de2021approximation} that there exists a tanh neural network $P_\epsilon :\mathbb{R}^{d(\lfloor \frac{\beta+1}{2}\rfloor)+1}\rightarrow\mathbb{R}$ (or Lemma A.3 in \cite{schmidt2020nonparametric} for ReLU) with depth  $O(1)$, width $O(1)$ and weights bounded by $O(\epsilon^{-C_\theta})$ such that
\begin{equation}\label{eq:nnaproxpoly}
\sup_{y\in [-C\epsilon^{-2\theta},C\epsilon^{-2\theta}]^{d\times\lfloor \frac{\beta+1}{2}\rfloor},t \in [0,1]}\| P_\epsilon(t,y)- P^{\lfloor \frac{\beta+1}{2}\rfloor}\big((y_k)_{k=0}^{\lfloor \frac{\beta+1}{2}\rfloor} \big)(t)\|\leq \epsilon^{\beta}.
\end{equation}
Define the functions \begin{align}
    \delta_i(t) := \left\{\begin{array}{ll} 1 & \text{ if } t\in [t_i,t_{i+1}]\\
    1-2t_i^{-1}(|t-t_i|\wedge |t-t_{i+1}|) & \text{ if } t\in [t_i/2,t_i]\cup [t_{i+1},t_{i+1}+t_i/2]\\
    0 & \text{ if } t\notin [t_i/2,t_{i+1}+t_i/2],
    \end{array}\right.
\end{align}
and
\begin{equation}\label{eq:xiii}
    \xi_i(t):=\frac{\delta_i(t)}{\sum_{j=0}^{m} \delta_j(t)},
\end{equation}
allowing to interpolate between the different approximations of $s(t,\cdot)$. 
Then, similarly to \eqref{eq:nnaproxpoly}, there exists a neural network $\bar{s}_i$ such that
$$
\sup_{x\in A_t^{\nu},t \in [0,1]}\|\bar{s}_i(t,x) - P_\epsilon(t,(N_i^k(x))_{k=0}^{\lfloor \frac{\beta+1}{2}\rfloor})\xi_i(t)\|\leq \epsilon^{\beta},
$$
and 
$$\sup_{x\in \mathbb{R}^d,t \in [0,1]} \|\bar{s}_i(t,x)\|\leq C\sup_{t\in [\epsilon^\theta,\epsilon^2], x\in A_t^{\nu}} \|s(t,x)\|, $$
having an architecture of depth $O(1)$, width $O(\epsilon^{-d})$ and weights bounded by $O(\epsilon^{-C_\theta})$.
Therefore, the neural network $$s_\epsilon(t,x)= \sum_{i=0}^m \bar{s}_i(t,x),$$
belongs to the class $\mathcal{F}$ composed of  neural networks having an architecture of depth $O(1)$, width $O(\log(\epsilon^{-C_\theta})\epsilon^{-d})$ and weights bounded by $O(\epsilon^{-C_\theta})$. In particular, the $\eta$-covering number of  $\mathcal{F}$ satisfies
\begin{align*}
    \log (\mathcal{N}(\mathcal{F},\|\cdot\|_\infty,\eta))\leq C_\theta \epsilon^{-d}\log((\epsilon\eta)^{-1})\log(\epsilon^{-1}).
\end{align*}
Furthermore, for $t\in[\epsilon^\theta,\epsilon^2]$ and $x\in A_t^{\nu}$ we have
\begin{align*}
    \|s(t,x)-s_\epsilon(t,x)\|\leq & \|s(t,x)-\sum_{i=0}^m P_\epsilon(t,(N_i^k(x))_{k=0}^{\lfloor \frac{\beta+1}{2}\rfloor})\xi_i(t)\|+\|\sum_{i=0}^m P_\epsilon(t,(N_i^k(x))_{k=0}^{\lfloor \frac{\beta+1}{2}\rfloor})\xi_i(t)-s_\epsilon(t,x)\|\\
    \leq & \sum_{i=0}^m \xi_i(t) \|s(t,x)-P_\epsilon(t,(N_i^k(x))_{k=0}^{\lfloor \frac{\beta+1}{2}\rfloor})\|+ \sum_{i=0}^m \| P_\epsilon(t,(N_i^k(x))_{k=0}^{\lfloor \frac{\beta+1}{2}\rfloor})\xi_i(t)-\bar{s}_i(t,x)\|\\
    \leq & C \log(\epsilon^{-1})^{C_2} \frac{\epsilon^{\beta}}{t}.
\end{align*}
Now, as $p_t$ is sub-Gaussian, we obtain for all $t\in [\epsilon^\theta,\epsilon^2]$ and $L>0$ that
\begin{align}\label{align:envrailuicestledernier}
\int \|s_\epsilon(x)-s(t,x)\|^2dp_t(x)= &\int \|s_\epsilon(t,x)-s(t,x)\|^2(\mathds{1}_{x\in A^{\nu}_t}+\mathds{1}_{x\notin A^{\nu}_t})dp_t(x)\nonumber\\
\leq & C \log(\epsilon^{-1})^{C_2}\frac{\epsilon^{2\beta}}{t}+\int \|s_\epsilon(t,x)-s(t,x)\|^2\mathds{1}_{x\notin A^{\nu}_t ,\|x\|\leq L\log(\epsilon^{-1}) }dp_t(x)\nonumber\\ 
&+ \int \|s_\epsilon(t,x)-s(t,x)\|^2\mathds{1}_{|x\|\geq L\log(\epsilon^{-1}) }dp_t(x)\nonumber\\
\leq & C_\theta \log(\epsilon^{-1})^{C_2} \frac{\epsilon^{2\beta}}{t},
\end{align}
taking $L>\theta$ large enough and recalling that $\nu = \epsilon^{2(\beta+\theta)}$ and $\|s(t,x)\|\leq Ct^{-1}(\|x\|+1).$

\end{proof}

\subsubsection{Proof of Corollary \ref{coro:approxnn2}}\label{sec:coro:approxnn2}
\begin{proof}
The proof follows the proof of Corollary \ref{coro:approxnn} with an additional time discretization.

Let us divide the interval $[\epsilon^2,\log(\epsilon^{-1})]$ into $[t_0,t_1], [t_1,t_2],...[t_{m-1},t_m]$ with $t_0=\epsilon^{2}$, $t_{i+1}=2t_i$ for $i\in \{0,...,m-2\}$ and $t_m=\log(\epsilon^{-1}).$
Additionally, taking
$$t_{i,j}=t_i+\frac{j}{L}(t_{i+1}-t_i) \text{ for } j\in \{0,...,L\}, \text{ and } L= \lfloor \epsilon^{-\delta/2}\rfloor,$$ 
we subdivide each interval $[t_i,t_{i+1}]$ into  $[t_{i,0},t_{i,1}], ...,[t_{i,L-1},t_{i,L}]$ with $t_{i,0}=t_i$ and $t_{i,L}=t_{i+1}$. 

\textbf{Approximation with polynomials of neural nets}\\
Let 
$$\gamma =2d\delta^{-1} \text{ and } \gamma_2=2\delta^{-1}(\beta+1),$$
and fix $R>0$ such that for all $t\geq \epsilon^{2}$ and $\nu = (\epsilon^{2(\beta+2)})\wedge 1/4$, we have $A^{\nu}_t\subset B(0,\log(\epsilon^{-1})^R)$. For $i\in \{0,...,m\}$ and $k\in \{0,...,\lfloor \gamma_2 \rfloor\}$, let $\mathcal{F}_i^{k}$ be the class of neural networks (given by either Theorem 5.1 in \cite{de2021approximation} or Theorem 5 in \cite{schmidt2020nonparametric}) with an architecture of depth $O(1)$, width $O(\epsilon^{-d-\delta/2})$ and weights bounded by $O(\epsilon^{-C_2})$ which satisfies that for all $f\in \mathcal{H}^{\beta+\gamma}_{C\log(\epsilon^{-1})^{C_2}t_i^{-(\frac{1}{2}(1+\gamma)+k)}}(B(0,\log(\epsilon^{-1})^R),\mathbb{R}^{d})$, there exists $N_f\in \mathcal{F}_{i}^k$ satisfying
$$\sup_{x\in B(0,\log(\epsilon^{-1})^R)}\|f-N_f\|\leq  C_\delta\log(\epsilon^{-1})^{C_2}\frac{\epsilon^{(1+\delta/(2d))(\beta+\gamma)}}{t_i^{\frac{1}{2}(1+\gamma)+k)}}.$$

From \eqref{eq:reguscorespace} we have that for all $k\in \mathbb{N}$, there exists $N_{i,j}^k\in \mathcal{F}_i^k$ such that
$$\sup_{x\in A_{t_{i,j}/2}^{\nu}}\|\partial_t^k s(t_{i,j},x)-N_{i,j}^k(x)\|\leq C_\delta\log(\epsilon^{-1})^{C_2}\frac{\epsilon^{(1+\delta/(2d))(\beta+\gamma)}}{t_i^{\frac{1}{2}(1+\gamma)+k)}}.$$
\paragraph{Exceptional case for certain times \( t_{i,j} \):}  
Recall from equations~\eqref{eq:dkdkoididjudueueueieoz} and~\eqref{eq:hgfjkzudijdjhdhdhhdhd} that the definition of \( A_t^\epsilon \) depends on whether \( t \) is above or below the threshold \( \log(\epsilon^{-1})^{-C^\star} \). For certain indices \( i \), there exist times \( t \in [t_{i,j}/2,\, t_{i,j+1} + t_{i,j}/2] \) such that \( A_{t_{i,j}/2} \neq A_t \). In this case, it must hold that \( t_i \geq \log(\epsilon^{-1})^{-C^\star}/8 \).

To avoid issues caused by the change in the support set \( A_t^\epsilon \) within the interval \( [t_{i,j}/2,\, t_{i,j+1} + t_{i,j}/2] \), we instead choose the approximating network \( N_{i,j}^k \) to satisfy the bound:
\[
\sup_{x \in A_\infty^\nu} \left\| \partial_t^k s(t_{i,j}, x) - N_{i,j}^k(x) \right\| \leq C_\delta \log(\epsilon^{-1})^{C_2} \frac{\epsilon^{(1 + \delta/(2d))(\beta + \gamma)}}{t_i^{\frac{1}{2}(1 + \gamma) + k}}.
\]

This is valid because as $t$ is sufficiently large, we retain the same regularity estimates (up to multiplicative constants) on the set \( A_\infty^\nu \), thanks to Proposition~\ref{prop:fjidididjhdhzpmwjf}, which states that for $x\in A_\infty^\epsilon$
\[
\left\| \nabla^k \left(s(t, \cdot) + \mathrm{Id}\right)(x) \right\| \leq C \log(\epsilon^{-1})^{C_2(1 + k)} \frac{e^{-(k+1)t}}{(1 - e^{-2t})^{k+1}}.
\]

Therefore, in the remainder of the analysis, we adopt the convention that \( A^\nu_{t_{i,j}/2} = A^\nu_\infty \) whenever \( t_i \geq \log(\epsilon^{-1})^{-C^\star}/8 \).

\textbf{Back to the approximation by polynomials of neural networks: }Defining for $t\in [t_{i,j},t_{i,j+1}]$:
\begin{align*}
    s_N(t,x) = P^{\lfloor \frac{\beta}{2} +\gamma_2\rfloor}\big((N_{i,j}^k(x))_{k=0}^{\lfloor \frac{\beta}{2}+\gamma_2\rfloor} \big)(t),
\end{align*}
we have from \eqref{eq:reguscoretime} that
\begin{align*}
    &\|s(t,x)-s_N(t,x)\|  \\
    & \leq \|s(t,x)-P^{\lfloor\frac{\beta}{2} +\gamma_2\rfloor}\big((\partial_t^k s(t_{i,j},x))_{k=0}^{\lfloor\frac{\beta}{2} +\gamma_2\rfloor}\big)(t)\| + \|P^{\lfloor\frac{\beta}{2} +\gamma_2\rfloor}\big((\partial_t^k s(t_{i,j},x))_{k=0}^{\lfloor\frac{\beta}{2} +\gamma_2\rfloor}\big)(t)-s_N(t,x)\|\\
    & \leq C_\delta\log(\epsilon^{-1})^{C_2} t_i^{-(1/2+\gamma_2)}(t_{i,j+1}-t_{i,j})^{\beta/2+\gamma_2}+C\|\sum_{k=0}^{\lfloor\frac{\beta}{2} +\gamma_2\rfloor} (\partial_t^k s(t_{i,j},x)-N_{i,j}^k(x))(t-t_{i,j})^k\|\\
    & \leq C_\delta\log(\epsilon^{-1})^{C_2}\left( t_i^{-(1/2+\gamma_2)}(t_i\epsilon^{\delta/2})^{\beta/2+\gamma_2}+\sum_{k=0}^{\lfloor\frac{\beta}{2} +\gamma_2\rfloor}\epsilon^{(1+\delta/(2d))(\beta+\gamma)}t_i^{-(\frac{1}{2}(1+\gamma)+k)}(t_i\epsilon^{\delta/2})^k\right)
    \\
    & \leq C_\delta\log(\epsilon^{-1})^{C_2}\left(t_i^{-1/2}\epsilon^{\beta+1}+\epsilon^{(1+\delta/(2d))(\beta+\gamma)}t_i^{-\frac{1}{2}(1+\gamma)}\right)\\
    & \leq C_\delta\log(\epsilon^{-1})^{C_2}t_i^{-1/2}\epsilon^{\beta+1}\left(1+\epsilon^{(\beta+\gamma)\delta/(2d)-1}\epsilon^\gamma t_i^{-\gamma/2}\right)\\
    & \leq C_\delta\log(\epsilon^{-1})^{C_2}t_i^{-1/2}\epsilon^{\beta+1}.
\end{align*}

\textbf{Construction of the final neural net}\\
This paragraph closely follows the one of the proof of Corollary \ref{coro:approxnn} in Section \ref{sec:coro:approxnn}. Define interpolation functions $\xi_{i,j}$ 
\begin{equation}\label{eq:xiii2}
    \xi_{i,j}(t):=\frac{\delta_{i,j}(t)}{\sum_{k,l} \delta_{k,l}(t)},
\end{equation}
with \begin{align}
    \delta_{i,j}(t) := \left\{\begin{array}{ll} 1 & \text{ if } t\in [t_{i,j},t_{i,j+1}]\\
    1-2t_{i,j}^{-1}(|t-t_{i,j}|\wedge |t-t_{i,j+1}|) & \text{ if } t\in [t_{i,j}/2,t_{i,j}]\cup [t_{i,j+1},t_{i,j+1}+t_{i,j}/2]\\
    0 & \text{ if } t\notin [t_{i,j}/2,t_{i,j}]\cup [t_{i,j+1},t_{i,j+1}+t_{i,j}/2].
    \end{array}\right.
\end{align}
Then, similarly to \eqref{eq:xiii}, we have that there exists a neural network $\bar{s}_{i,j}$ such that
$$
\sup_{x\in A_t^{\nu},t \in [0,1]}\|\bar{s}_{i,j}(t,x) - P(t,(N_i^k(x))_{k=0}^{\lfloor \frac{\beta}{2}+\gamma_2\rfloor})\xi_{i,j}(t)\|\leq \epsilon^{\beta+1},
$$
and 
$$\sup_{x\in \mathbb{R}^d,t \in [0,1]} \|\bar{s}_i(t,x)\|\leq C\sup_{t\in [\epsilon^2,\log(\epsilon^{-1})], x\in A_t^{\nu}} \|s(t,x)\|, $$
having an architecture of depth $O(1)$, width $O(\epsilon^{-d-\delta/2})$ and weights bounded by $O(\epsilon^{-C_\delta})$.
Therefore, the neural network $$s_\epsilon(t,x)= \sum_{i=0}^m\sum_{j=0}^{L} \bar{s}_{i,}(t,x),$$
belongs to the class $\mathcal{F}$ composed of  neural networks having an architecture of depth $O(1)$, width $O(\log(\epsilon^{-1})\epsilon^{-d-\delta})$ and weights bounded by $O(\epsilon^{-C_\delta})$. In particular, the $\eta$-covering number of  $\mathcal{F}$ satisfies
\begin{align*}
    \log (\mathcal{N}(\mathcal{F},\|\cdot\|_\infty,\eta))\leq C_\theta \epsilon^{-d-\delta}\log((\epsilon\eta)^{-1})\log(\epsilon^{-1}).
\end{align*}
Now, doing as \eqref{align:envrailuicestledernier}, we obtain for all $t\in [\epsilon^2,\log(\epsilon^{-1})]$
\begin{align*}
    \int \|s_\epsilon(t,x)-s(t,x)\|^2dp_t(x)\leq & C_\delta \log(\epsilon^{-1})^{C_2} \frac{\epsilon^{2(\beta+1)}}{t}.
\end{align*}
\end{proof}

\end{document}